\documentclass[11pt]{article}

\usepackage{epsfig,fancyhdr,color}

\usepackage{graphicx}
\definecolor{NoteColor}{rgb}{1,0,0}

\headsep .in
\usepackage{amsmath}
\usepackage{amssymb}

\numberwithin{equation}{section}

\newtheorem{thm}{Theorem}[section]
\newtheorem{defi}[thm]{Definition}
\newtheorem{prop}[thm]{Proposition}
\newtheorem{lem}[thm]{Lemma}
\newtheorem{cor}[thm]{Corollary}
\newtheorem{con}[thm]{Conjecture}

\newtheorem{prob}[thm]{Problem}
\newtheorem{rem}[thm]{Remark}

\def\M{\mathcal M}

\def\a{\mathfrak a}

\def\Mod{\mathrm Mod}

\def\ga{\Gamma}
\def\lsp{\Gamma\backslash X}

\def\R{\mathbb R}

\def\Z{\mathbb Z}

\def\C{\mathbb C}
\def\M{\mathcal M}
\def\T{\mathcal T}
\def\Mod{\mathrm {Mod}}
\def\Diff{\mathrm {Diff}}
\def\Out{\mathrm {Out}}
\def\ga{\Gamma}

\def\T{\mathcal T}

\title{Complete invariant geodesic metrics on outer spaces and Jacobian varieties of tropical curves}
\author{Lizhen Ji\thanks{Partially Supported by NSF grant  DMS-1104696}
\\Department of Mathematics\\ University of Michigan\\ Ann Arbor, MI 48109}

\date{November 8, 2012}

\begin{document}

\maketitle

\begin{abstract}
Let $\mathrm{Out}(F_n)$ be the outer automorphism group of the free group $F_n$.
It acts properly on the outer space $X_n$ of marked metric graphs, which is a finite-dimensional 
infinite simplicial complex with some simplicial faces missing.
In this paper, we construct  complete geodesic metrics and complete piecewise
smooth Riemannian metrics on $X_n$
which are invariant under $\mathrm{Out}(F_n)$.
One key ingredient is the identification of metric graphs with tropical curves
and  the use of the tropical Jacobian map from the moduli space of tropical curves
to the moduli space of  principally polarized tropical abelian varieties. 
\end{abstract}

\tableofcontents

\section{Introduction}

Let $F_n$ be the free group on $n$ generators, $n\geq 2$,
and $\mathrm{Out}(F_n)=\mathrm{Aut}(F_n)/\mathrm{Inn}(F_n)$ its outer automorphism group.
The group $\text{Out}(F_n)$ is one of the most basic groups in combinatorial group theory 
and it has been
extensively studied. One reason is that it is related to a basic question: given a finitely
generated group $\ga$, how to get all possible generating sets of $\ga$ \cite[p. 120]{maks}.
One key ingredient in understanding the properties of $\mathrm{Out}(F_n)$
is to use its action on the outer space $X_n$ of marked metric graphs and the properties
of $X_n$.

According to the celebrated Erlangen program of Klein, 
an essential part of the geometry of a space is concerned with invariants of the isometries
(or symmetries) of the space. Similarly, an essential part of the geometry of a group is to find and
understand spaces on which the group acts and preserves suitable additional structures
of the space.
It is often natural to require that the space is a metric space and the action is by isometries.

In classical geometry,  a metric space is a complete Riemannian manifold 
such as the Euclidean
space, the sphere or the hyperbolic space.
But it is also important to consider metric spaces which are not manifolds.
Examples include Tits buildings and Bruhat-Tits buildings for real
and $p$-adic semisimple Lie groups. These are simplicial complexes
with a natural complete Tits metric so that the groups act isometrically and simplicially on them. 
Of course, the rich combinatorial structure also makes their geometry interesting. 

The outer space $X_n$  is a finite-dimensional infinite simplicial complex and hence admits a 
natural simplicial metric $d_0$ (Proposition \ref{simplicial-metric}), 
where each $n$-simplex is realized as 
the standard simplex in $\R^{n+1}$ with vertices $(1, 0, \cdots, 0), \cdots, (0, \cdots, 0, 1)$
 and given the metric induced from the metric of $\R^{n+1}$. 
Clearly, $d_0$ is invariant under $\mathrm{Out}(F_n)$. Since simplices of $X_n$ have
missing simplicial faces, $d_0$ is not a complete metric on $X_n$. 

Since the completeness of metrics is a basic condition which is
important for many applications,
one  natural problem is the following (see  \cite[Question 2]{briv} for more discussion).

\begin{prob}\label{prob-1.1}
Construct  complete geodesic metrics on the outer space $X_n$ which are invariant
under $\mathrm{Out}(F_n)$. 
\end{prob}

As explained above, solutions to this problem provide geometries of $\mathrm{Out}(F_n)$
in a certain sense. 
Another motivation for this problem
comes from the analogy with other important groups in geometric group
theory:
arithmetic subgroups $\ga$ of semisimple Lie groups $G$ and mapping class groups
$\Mod_{g,n}$ of surfaces of genus $g$ with $n$ punctures. 
A lot of work on $\mathrm{Out}(F_n)$ is motivated by results
obtained for these families of groups.

An arithmetic group $\ga$ acts on the symmetric space $G/K$,
and the mapping class group 
$\Mod_{g,n}$ acts on the Teichm\"uller space $\T_{g,n}$ of Riemann surfaces
of genus $g$ with $n$ punctures.
The outer space $X_n$ was introduced as a space on which $\mathrm{Out}(F_n)$ acts naturally.
These actions have played a crucial and similar role in understanding properties
of the groups.

Though these three classes of groups and spaces share many common properties,
there is one important difference between them, 
which might have not been emphasized enough:
the symmetric spaces $X$ and  the Teichm\"uller space $\T_{g,n}$ are smooth manifolds,
but the outer space $X_n$ is not a manifold. 

The symmetric space $X=G/K$ admits a complete $G$-invariant Riemannian metric 
and $\ga$ acts isometrically and properly on $X$. It is known that
$X$ is a contractible, nonpositively curved 
Riemannian manifold, and it is an important example of a Hadamard manifold.
Furthermore, the quotient $\lsp$ has finite volume.

The Teichm\"uller space $\T_{g,n}$ admits several complete Riemannian and Finsler 
metrics such as the Bergman and Teichm\"uller metrics, 
and $\Mod_{g,n}$ acts isometrically and properly on them. 
$\T_{g,n}$  also admits an incomplete negatively curved (Weil-Petersson) metric.
For all these natural metrics, the quotient $\Mod_{g,n}\backslash \T_{g,n}$
has finite volume.  
An interesting feature is that the known complete metrics on $\T_{g,n}$ are not nonpositively
curved, but the negatively curved Weil-Petersson metric is incomplete.
A natural and folklore conjecture is that {\em $\T_{g,n}$ does not admit
any complete nonpositively curved metric which is invariant under $\Mod_{g,n}$.}

Though $X_n$ is not a manifold, it
 is a locally finite simplicial complex and hence is a canonically stratified space
with a smooth structure in the sense of \cite{pf}.
Therefore, the following problem also seems to be  natural in view of the above analogy.

\begin{prob}\label{prob-1.2}
Construct  piecewise smooth Riemannian metrics on the outer space $X_n$ which are
 invariant under $\mathrm{Out}(F_n)$ and whose induced length metrics
 are complete geodesic metrics. Furthermore,
the quotient $\mathrm{Out}(F_n)\backslash X_n$ has finite volume.
\end{prob}

With respect to the smooth structure on $X_n$ as a canonically stratified space,
it is natural to require that the Riemannian metric on $X_n$ is smooth in the sense of \cite[Chapter 2]{pf}. 
If the second condition of finite volume 
is satisfied, it might give further justification to the {\em philosophy
that $\mathrm{Out}(F_n)$ is an analogue
of lattices of Lie groups}, which has played a guiding role
in understanding $\mathrm{Out}(F_n)$.

In this paper, we construct several invariant complete geodesic metrics and complete 
piecewise-smooth Riemannian metrics on $X_n$ and hence solve 
 Problem \ref{prob-1.1} and 
Problem \ref{prob-1.2}.
The basic idea is to identify metric graphs with tropical curves and use the
tropical  Jacobian map 
of tropical curves together with the simplicial metric $d_0$ on the outer space $X_n$.
For precise statements, see Theorems \ref{main1}, \ref{main2} and \ref{main4},
and Propositions \ref{finite-vol-1}, \ref{finite-vol-2} below. 
In the process, we also clarify some facts on the moduli space of tropical abelian varieties
(Proposition \ref{moduli-abelian}).

\begin{rem}
{\em It should be pointed that the tropical Jacobian map, or rather the period map of metric graphs,  
can be defined directly for compact metric
graphs. But without interpretation through tropical curves and their Jacobians,
it is a mystery why it is defined in this way. There are infinitely many different ways
to define a map from the outer space $X_n$ to the symmetric space 
$\mathrm{GL}(n, \R)/\mathrm{O}(n)$ of positive definite matrices
(see Remark \ref{many-def} below). In any case, the connection between metric graphs
and tropical curves is interesting and our discussion here also
clarifies some questions in tropical geometry. It is expected that
the geometry of the outer space will be helpful to the study of tropical curves and their moduli spaces. See Remark \ref{trop-moduli} below.
}
\end{rem}

\begin{rem}
{\em 
It is known that $X_n$ admits a non-symmetric geodesic metric, called
the Lipschitz (or Thurston)  metric. Briefly, the non-symmetric
 Lipschitz metric is not complete, and its associated (symmetric) is not geodesic.
For its definition, some properties and applications,
see \cite{frm} \cite{alb} \cite{al}. 
}
\end{rem}

\begin{rem}
{\em 
It seems reasonable to conjecture
 that  $X_n$ does not admit any complete CAT(0)-metric 
which is invariant under $\mathrm{Out}(F_n)$.
For example, it was shown in \cite{bri} that the standard spine of $X_n$ does not admit any
invariant piecewise-Euclidean or piecewise-hyperbolic metric of nonpositive curvature.
It was shown in \cite{ger}  \cite{briv1} 
that $\mathrm{Out}(F_n)$ is not a CAT(0)-group, i.e., it cannot act
isometrically and cocompactly on a CAT(0)-space.
Since the quotient $\mathrm{Out}(F_n)\backslash X_n$ is not compact, the
above result does not exclude the possibility
that $X_n$ might admit a CAT(0)-metric which is invariant
under $\mathrm{Out}(F_n)$. 
}
\end{rem}

\vspace{.1in}

\noindent{\em Acknowledgments.} The idea of using the period map of metric graphs
in Equation (\ref{period-map-1})
to pull back and define a metric on the outer space
was suggested by Enrico Leuzinger, partly motivated by the work
\cite{jl},  and I would like to thank him for sharing this idea.
Though the period map fails to be injective (unlike the case of Riemann surfaces), the
simple idea of combining the period map with the canonical simplicial metric
$d_0$ on the outer space $X_n$ is crucial to overcome this
 problem. I would also like to thank Mladen Bestvina for helpful information
 about CAT(0)-metrics on the outer space, and Athanase Papadopoulos for carefully reading
 a preliminary version of this paper. 

\section{Definitions and basic facts on $\mathrm{Out}(F_n)$ and the outer space $X_n$}

In this section, we recall some basic definitions and results on the outer automorphism
group $\mathrm{Out}(F_n)$ and the outer space $X_n$, and relations with arithmetic groups and
mapping class groups.

Probably the most basic arithmetic subgroup of a semisimple Lie group
is $ \mathrm{SL}(n, \Z)\subset  \mathrm{SL}(n, \R)$.
A closely related group is the arithmetic subgroup $\mathrm{GL}(n, \Z)$
of a reductive Lie group $\mathrm{GL}(n, \R)$, which contains $\mathrm{SL}(n, \Z)$ as a subgroup
of index 2.

The group $\mathrm{Out}(F_n)$ is related to $\mathrm{GL}(n, \Z)$ as follows. 
Since the abelinization of $F_n$ is $\Z^n$ and $\Out(\Z^n)=\mathrm{Aut}(\Z^n)=\mathrm{GL}(n, \Z)$, 
there is a homomorphism $A: \mathrm{Out}(F_n)\to \mathrm{GL}(n, \Z)$.
It is known that for every $n\geq 3$, 
the homomorphism $A: \mathrm{Out}(F_n)\to \mathrm{GL}(n, \Z)$ is {\em surjective with an infinite kernel}
and that when $n=2$, $A: \Out(F_2)\to \mathrm{GL}(2, \Z)$ is an isomorphism.
See \cite[p. 2]{vog} for references.

Let $S_g$ be a compact oriented surface of genus $g$,
$\Diff^+(S_g)$ its group of orientation preserving diffeomorphisms,
and $\Diff^0(S_g)$ its identify component.
Then the quotient group
$\Mod_g=\Diff^+(S_g)/\Diff^0(S_g)$ is the {\em mapping class group} of $S_g$.

By the Dehn-Nielson theorem, $\Mod_g=\Out(\pi_1(S_g))$.
Since the abelinization of $\pi_1(S_g)$ is $H_1(S_g, \Z)$,  which is isomorphic to $\Z^{2g}$,
we also get a homomorphism
$\Mod_{g}\to \mathrm{GL}(2g, \Z)$.
Since elements of $\Mod_g$ also preserve the intersection form
on $H_1(S_g, \Z)$, which is a symplectic form, the image of $\Mod_{g}$ lies in 
$ \mathrm{Sp}(2g, \Z)$. Therefore, we have a homomorphism 
$$j: \Mod_{g}\to  \mathrm{Sp}(2g, \Z).$$
It is known that $j$ is surjective with an infinite kernel when $g\geq 2$, and 
$j$ is an isomorphism when $g=1$, i.e., $\Mod_1\cong \mathrm{SL}(2, \Z)$.


The above discussions 
establish close relations between $\mathrm{Out}(F_n)$, and the mapping class group $\Mod_g$
and arithmetic groups such as $\mathrm{GL}(n, \Z)$ and $\mathrm{Sp}(2g, \Z)$.
These groups share many common properties, for example, group theoretical
properties and cohomological properties. 
One important approach to study groups is to  study their actions on suitable spaces.
But it is often difficult to find spaces with finiteness properties.

As mentioned before,  an arithmetic subgroup $\ga \subset G$
 acts isometrically and properly
on the symmetric space $X=G/K$, 
 and $X$ is a finite dimensional model of the universal space $\underline{E}\ga$
for proper actions of $\ga$.

An important space for $\Mod_{g,n}$ to act on
is the Teichm\"uller space $\T_{g,n}$, which is the space of marked
complex structure of $S_{g,n}$, i.e., marked Riemann surfaces of genus $g$ with $n$
punctures. The action  of $\Mod_{g,n}$ on $\T_{g,n}$
comes from changing the markings.  
It is an important result of Teichm\"uller that the Teichm\"uller space $\T_{g,n}$ is a
contractible manifold and that the action is proper and isometric with respect
to the Teichm\"uller metric (and also many other metrics). 
It is also known that $\T_{g,n}$ is a finite dimensional model of $\underline{E}\Mod_g$.

Motivated by Teichm\"uller space, Culler and Vogtmann \cite{cv}
 introduced the outer space $X_n$
of marked metric graphs.

A metric graph $(G, \ell)$ is a graph with a positive length $\ell(e)$ assigned to each edge $e$.
Note that the fundamental group of a finite
graph is a free group. 
We only consider graphs $G$ with genus equal to $n$, i.e., $\pi_1(G)\cong F_n$
(or rather $H_1(G, \Z)\cong \Z^n$).
We also assume that $G$ {\em has no vertex of valence 1 or 2, and no separating
edge (or bridge).} 

Let $R_n$ be the wedge of $n$ circles, also called {\em a rose with $n$ petals.}
Then {\em a marking} on a metric graph 
$(G, \ell)$ is an homotopy equivalence $h: G \to R_n$, which amounts
to  a choice
of a set of generators of $\pi_1(G)$. {\em A marked metric graph} is a metric graph
$(G, \ell)$ 
together with a marking $h$,  denoted
by $(G, \ell, h)$.

Two marked metric graphs $(G, \ell, h), (G', \ell',  h')$ are called 
{\em equivalent} if there exists an
isometry between $(G, \ell)$ and $(G', \ell')$ that commutes with the markings $h, h'$ up to homotopy.

Usually we {\em normalize} (or {\em scale}) the metric $\ell$ of each metric graph  $(G, \ell)$
so that the sum of lengths of all edges is equal to 1. (In Remark \ref{trop-moduli}, we will
also use metric graphs whose edge lengths are not normalized.)

Then the {\em outer space} (or {\em the reduced outer space})
$X_n$ is defined to be the set of {\em equivalence classes of marked metric graphs}
of total length 1,
$$X_n=\{(G, \ell, h)\}/\sim.$$

Since elements of $\mathrm{Out}(F_n)$ can be realized by homotopy equivalences of $R_n$,
$\mathrm{Out}(F_n)$ acts on $X_n$ by changing the markings of metric graphs.

\begin{prop} The outer space $X_n$ is a locally finite simplicial complex 
of dimension $3g-4$, and its simplices are partially open, i.e.,
some simplicial faces of the simplices are missing. 
\end{prop}

Briefly, for each marked graph $(G, h)$, let 
$e_1, \cdots, e_k$ be its edges. Then variation of the edge lengths
$\ell(e_1), \cdots, \ell(e_k)$ under the conditions $\ell(e_1), \cdots,
\ell(e_k)>0$, $\sum_{i=1}^k \ell(e_i)=1$ fills out an {\em open}
$(k-1)$-simplex $\Sigma_{(G, h)}$.
Its simplicial faces correspond to vanishing of some edge lengths $\ell(e_i)$.
For any subset $e_{i_1}, \cdots, e_{i_r}$ of edges, the simplicial face of $\Sigma_{(G, h)}$
defined by $\ell(e_{i_1})=\cdots =\ell(e_{i_r})=0$ is contained in the outer space $X_n$
if and only if no loop of $G$, i.e., a nontrivial element of $\pi_1(G)$,
is contained in the union of these edges $e_{i_1}, \cdots, e_{i_r}$. (Note that if an edge length
$\ell(e)=0$, then the edge $e$ is removed from the graph, and if the total length of a loop
is equal to zero, the genus of the resulting graph will decrease.)

The outer space $X_n$ is the union of  open simplexes $\Sigma_{(G, h)}$,
and its maximal dimensional simplices correspond to trivalent graphs $(G, h)$.
There are several ways to put a topology on $X_n$.
Probably the most direct way is to glue up the simplicial topology of these simplices.
See \cite{vog} for more details and references on this result and several others recalled below.

An important result due to Culler-Vogtmann \cite{cv} is

\begin{thm}\label{cont}
The outer space $X_n$ is contractible. 
\end{thm}

Clearly, we can give each $k$-dimensional simplex $\Sigma_{(G,h)}$ 
the {\em standard simplicial metric}
by identifying it with the standard simplex in $\R^{k+1}$,
$\{(x_1, \cdots, x_{k+1})\in \R^{k+1}\mid x_1, \cdots, x_k >0, x_1+\cdots + x_{k+1}=1\}$,
whose edges have  length $\sqrt{2}$.
This metric extends and defines a metric on the closure of $\Sigma_{(G,h)}$
in $X_n$. It is also clear that these metrics are compatible on their common
simplicial faces. 
Therefore, it defines a {\em length function}, or a {\em metric} on $X_n$, denoted by $d_0$. 

\begin{prop}\label{simplicial-metric}
The length metric $d_0$ on $X_n$ is an incomplete metric on $X_n$
which is invariant under $\mathrm{Out}(F_n)$, 
and its restriction to each simplex
$\Sigma_{(G, h)}$ is the length metric induced from a Riemannian metric.
\end{prop}
\begin{proof}
Since all top dimensional simplices in $X_n$ are isometric, 
it is clear that the restriction to each simplex $\Sigma_{(G, h)}$ is the length metric
of the flat Riemannian metric of the ambient Euclidean space $\R^k$,
i.e., for every two points in $\Sigma_{(G, h)}$,
there is no shortcut by going through neighboring simplices.
Since the closure of each open simplex $\Sigma_{(G, h)}$ in $X_n$ is noncompact
and its diameter is equal to $\sqrt{2}$,
$d_0$ is an incomplete metric. Since the action of $\mathrm{Out}(F_n)$ on $X_n$ preserves
the metrics of the simplices, the invariance of $d_0$ under $\mathrm{Out}(F_n)$ is clear.
\end{proof}

See \cite[\S 1]{bri} for related discussion on how to glue up metrics on simplices. 

\begin{rem}
{\em It is natural to conjecture that 
the length metric $d_0$ on $X_n$ is a {\em geodesic metric}
and is also {\em geodesically convex} in the sense that for every two points
$p, q\in X_n$, there exists a unique geodesic segment connecting them.
This is stronger than the result in Theorem \ref{cont} on the contractibility of $X_n$.
If this conjecture is true, then $d_0$ is a good analogue of the Weil-Petersson metric on  
Teichm\"uller space. Indeed, though the Weil-Petersson metric
of  Teichm\"uller space is incomplete,
the Weil-Petersson metric is geodesically convex  according to a result of Wolpert \cite{wol87}.
One might try to use this metric to understand elements
of $\mathrm{Out}(F_n)$ as the Weil-Petersson metric was used in the paper \cite{daw}.
On the other hand, the situation is more complicated since 
axes of fully irreducible outer automotphisms are  not unique in general
\cite{bes} \cite{ham1} \cite{ham2}.
}
\end{rem}

\begin{prop}\label{finite-cell} 
Under the action of $\mathrm{Out}(F_n)$ on $X_n$, there are only finitely many orbits of 
open simplices $\Sigma_{(G, h)}$.
\end{prop}

Briefly, since the quotient $\mathrm{Out}(F_n)\backslash X_n$ is the moduli space
of metric graphs of genus $n$,
 the $\mathrm{Out}(F_n)$-orbits of open simplices in $X_n$ correspond to 
equivalence classes of finite graphs of genus $n$
with only vertices of valence at least 3 and no separating edges.
Clearly, there are only finitely many such graphs up to equivalence. 

\begin{prop}
The group $\mathrm{Out}(F_n)$ acts properly on the outer space $X_n$.
\end{prop}

This follows from the fact that for every simplex $\Sigma_{(G,h)}$ in
$X_n$, its stabilizer in $\mathrm{Out}(F_n)$ is finite (in fact, isomorphic to
a subgroup of the symmetry group $S_N$, where $N$ is the number of vertices
of the graph $G$), and $X_n$ is a locally finite simplicial
complex.

The combination of the above results suggests the following \cite{wh} \cite{krv}.

\begin{prop}
$X_n$ is a universal space for proper actions of $\mathrm{Out}(F_n)$, i.e.,
for every finite subgroup $H\subset F_n$, the set of fixed points
$X_n^H$ of $H$ is nonempty and contractible. 
\end{prop}



\section{A natural approach to define a metric}

One natural method to construct an invariant metric on $X_n$ is to embed $X_n$ equivariantly
into a metric space $(Z, d)$ and then pull back the  metric $d$.
If $(Z, d)$ is a complete metric space and the image $i(X)$ is a complete metric subspace,
then the pulled-back metric $i^*d$ is an invariant complete metric on $X_n$.
On the other hand, even if $(Z, d)$ is a geodesic metric, $i^*d$ may not be a geodesic metric.

If $X_n$ were a manifold, and $(Z, d)$ is a Riemannian manifold,
then an immersion $i: X_n\to Z$, not necessary an embedding,
will allow us to pull back the Riemannian metric of $Z$
to obtain a Riemannian metric on $X_n$ and consequently a Riemannian
distance.

One difficulty is that there is no obvious choice of
{\em complete metric spaces} $(Z, d)$ into which $X_n$ can be {\em embedded equivariantly}.
 (We note that there is a natural embedding
$X_n\to Y$ into a topological space such that the closure is compact,
but the ambient space $Y$ is not a metric space. 
See \cite{cv2} \cite[\S 1.4]{vog} for details and references.)

Instead, we will look for a complete metric space $(Z, d)$ and an equivariant map
$i: X_n\to (Z, d)$ satisfying the condition: for any sequence of points $x_k\in X_n$
which converges to a boundary point as $k\to \infty$, 
i.e., a point in a missing simplex, then for every
basepoint $x_0\in X_n$, $d(i(x_k), i(x_0))\to +\infty$  as $k\to \infty$.
In some sense, we give up the condition of embedding and retain the completeness
condition on the image. By combining it with the incomplete simplicial metric $d_0$ on $X_n$,
we can construct complete geodesic metrics on $X_n$ which are invariant under $\mathrm{Out}(F_n)$.


One such map can be constructed by borrowing ideas from Riemann surfaces, Teichm\"uller
spaces, and symmetric spaces.
Note that $\mathrm{GL}(n, \R)$ is a reductive Lie group,  $\mathrm{O}(n)$ is a maximal compact subgroup
of $\mathrm{GL}(n, \R)$, and the quotient $\mathrm{GL}(n, \R)/ \mathrm{O}(n)$ with a $\mathrm{GL}(n, \R)$-invariant metric 
is a symmetric space of nonpositive curvature, though not of noncompact type.

\begin{prop}\label{iden}
The symmetric  space $X=\mathrm{GL}(n, \R)/\mathrm{O}(n)$  can be identified
with the space of positive definite quadratic forms,
or equivalently, the space of marked compact flat tori of dimension $n$.
\end{prop}
\begin{proof}
It is clear that $\mathrm{GL}(n, \R)$ acts on the space of positive definite $n\times n$-matrices:
for $g\in \mathrm{GL}(n, \R)$ and a positive definite matrix $A$,
$g\cdot A=g A g^t$.
By the spectral decomposition of positive definite quadratic forms, we can prove that this action
is transitive. Since the stabilizer of the identity matrix is equal to $ \mathrm{O}(n)$, the first statement
is proved.

As in the case of marked graphs, a {\em marked flat torus} of dimension $n$
is a compact flat Riemannian torus $M^n$ with a diffeomorphism $\varphi:
M^n\to \Z^n\backslash \R^n$, where $\varphi$ is well-defined  up to homotopy.
To prove the second statement, we note that 
every positive definite quadratic form on $\R^n$ defines a flat metric on $\R^n$
which descends to a flat metric on the torus $\Z^n\backslash \R^n$.
This torus has the marking giving by the identity map $\Z^n\backslash \R^n\to 
\Z^n\backslash \R^n$. Conversely, every marked flat torus of dimension $n$ 
canonically induces a flat metric on $\Z^n\backslash \R^n$, which is in turn
induced  by a unique positive definite quadratic form on $\R^n$.  
\end{proof}

Clearly, $\mathrm{GL}(n, \Z)$ acts on the symmetric space $X$, and the quotient $\mathrm{GL}(n, \Z)\backslash X$
is the isometry classes of compact flat tori of dimension $n$.

We can define an {\em equivariant map} 
$$\Pi: X_n\to X$$
with respect to the natural actions of $\mathrm{Out}(F_n)$ and $\mathrm{GL}(n, \Z)$,
and the homomorphism $\mathrm{Out}(F_n) \to \mathrm{GL}(n, \Z)$
as follows.

For every marked metric graph $(G, \ell, h)$ in $X_n$,
there is a canonical isomorphism $h_*: H_1(G, \Z)\cong \Z^n$,
and  hence an isomorphism $h_*: H_1(G, \R)\to \R^n$ by linear extension.
Therefore $H_1(G, \Z)\backslash H_1(G, \R)$ is a marked compact torus.
We will use the metric $\ell$ on the graph to define a flat metric on the torus,
or equivalently, a positive definite quadratic form $Q$ on the real vector space
$H_1(G, \R)$.
Note that $G$ is of dimension 1
and  $ H_1(G, \R)$ is equal to the set of 1-cycles on $G$.
Therefore, it suffices to define
a quadratic form $Q$ on the set $C_1(G, \R)$ of all 1-chains on $G$ and then to
restrict it to the subspace of 1-cycles.
 
Fix an orientation for every edge $e$ of $G$.
Then every 1-chain $\sigma$ on $G$ can be written uniquely as $\sum_{e\in G} a_e e$,
where $a_e\in \R$. 
For every two edges $e, e'$ of $G$, define
\begin{equation}\label{quad-form}
Q(e, e')=\begin{cases} \ell(e), & \quad \text{\ if\ } e=e',\\
0, & \text{\ if \ } e\neq e'.
\end{cases}
\end{equation}
By extending it linearly, we obtain a quadratic form $Q$ on $C_1(G, \R)$ \cite[p. 151]{cvi}.

\begin{prop}\label{pos-definite}
The bilinear quadratic form $Q$ on  $C_1(G, \R)$,  and hence its restriction on the subspace $H_1(G,
\R)$,  is positive definite. 
\end{prop} 
\begin{proof}
For any  nontrivial 1-chain $\sigma$ of $G$, write $\sigma=\sum_{e\in G}
a_e e$,  where $a_e\neq 0$ for at least one edge. It follows from the definition that 
$$Q(\sum_{e \in G} a_e e , \sum_{e\in G} a_e e)=\sum_{e\in G} a_e^2 \ell(e)>0.$$
\end{proof}

In terms of the canonical identification $H_1(G, \Z)\cong \Z^n$ and $H_1(G, \R)\cong \R^n$,
the quadratic form $Q$ becomes a positive definite matrix of size $n\times n$.
It is called the {\em period matrix} of the marked metric graph $(G, \ell, h)$.
By Proposition \ref{iden}, we obtain a {\em period map}
\begin{equation}
\Pi: X_n\to X=\mathrm{GL}(n, \R)/\mathrm{O}(n), \quad (G, \ell, h) \mapsto  Q,
\end{equation}
which is clearly equivariant with respect to the homomorphism $\mathrm{Out}(F_n)\to \mathrm{GL}(n, \Z)$.

By Proposition \ref{iden} again, this positive definite quadratic form $Q$ induces a flat
metric on the marked torus $H_1(G, \Z)\backslash H_1(G, \R)$.
To be consistent with the Jacobian map of tropical curves below, we denote the
period map by 
\begin{equation}\label{period-map-1}
\Pi: X_n\to X=\mathrm{GL}(n, \R)/ \mathrm{O}(n), 
\quad (G, \ell, h) \mapsto (H_1(G, \Z)\backslash H_1(G, \R), Q),
\end{equation}
where the image consists of marked compact flat tori.

We will explain below
why it is called the period map and compare it with the period (or Jacobian) map
of Riemann surfaces. 

The definition of the quadratic form $Q$ in Equation \ref{quad-form}
might look puzzling, and may seem  wrong.
The reason is that the length of the chain $e$ is equal to $\sqrt{\ell(e)}$, which is not
a smooth function of the lengths of the metric graph $(G, \ell)$.  
From the above proof, it is clear that if we use $\ell(e)^2$ instead of $\ell(e)$
and define
$\tilde{Q}(e, e)=\ell(e)^2,$ then we still obtain a positive definite quadratic form on
the torus $H_1(G, \Z)\backslash H_1(G, \R)$.
This latter definition might look more natural, since the length of a chain $e$
would be equal to $\ell(e)$, which 
will depend smoothly on the metric $\ell$ of the metric graph $(G, \ell)$. 


\begin{rem}\label{many-def}
{\em The proof of Proposition \ref{pos-definite} shows that for every function $f(x)$ satisfying
$f(x)>0$ for $x>0$,
we can obtain a positive quadratic form $Q_f$ on $H_1(G, \R)$ by defining
$Q_f(e, e)=f(\ell(e))$ as in Equation \ref{quad-form}
 and hence an equivariant map $\Pi_f: X_n\to \mathrm{GL}(n, \R)/\mathrm{O}(n)$. 

To explain that the function $f(x)=x$ is the right choice, we need to go through tropical geometry.
The definition of the quadratic form $Q$ in Equation \ref{quad-form}
 turns out to be the right definition when $(H_1(G, \Z)\backslash H_1(G, \R), Q)$ is viewed
as  a {\em principally polarized tropical abelian variety}  of a  marked 
{\em tropical curve}, when the  metric graph $(G, \ell)$ is identified with a tropical
curve.  
Another potentially important point is that $Q$ is a {\em linear function} when it is restricted to
each simplex of the outer space $X_n$ and hence it has the characteristic property
of tropical maps in tropical geometry, which might imply that {\em the tropical Jacobian map
is a tropical map} \cite[Assumption 3, p. 165]{cvi}. 
To explain these,
we will describe some basic notions in tropical geometry in the next sections.

Another reason for discussing tropical geometry is that discussion in this paper also
clarifies some points on tropical curves and the moduli space of principally polarized
tropical abelian varieties (Corollary \ref{moduli-abelian}).
 It seems reasonable to believe that the geometry
of the outer space  $X_n$ will be potentially useful in studying tropical curves and their moduli spaces.
 See 
Remark \ref{trop-moduli} below.
}
\end{rem}

To define a metric on $X_n$,
there is a further problem with using the period map $\Pi$ in Equation \ref{period-map-1}. 
When $n\geq 3$, 
the homomorphism $\mathrm{Out}(F_n)\to \mathrm{GL}(n, \Z)$ has an infinite kernel.
Since 
the map $\Pi$ is equivariant with respect to the homomorphism $\mathrm{Out}(F_n)\to \mathrm{GL}(n, \Z)$,
it  cannot be injective. 
But as we will show below (Corollary \ref{failure}), $\Pi$ is not even injective locally
 when $n\geq 3$, i.e., the induced quotient map $\Pi: \mathrm{Out}(F_n)\backslash X_n\to 
 \mathrm{GL}(n, \Z)\backslash \mathrm{GL}(n, \R)/\mathrm{O}(n)$ is not injective,  unlike 
 the period map of Riemann surfaces in Equation \ref{torelli}.
When it is restricted to each simplex in $X_n$, the period map $\Pi$ is a differential map, but it is
not injective or is an immersion in general. 
Though $\mathrm{GL}(n, \R)/\mathrm{O}(n)$ is a symmetric space and
has an invariant metric, we cannot use this map to pull back and define
a metric on $X_n$. 

To overcome this difficulty, one simple and crucial 
idea is to make use of  the simplicial metric $d_0$ on $X_n$ 
in Proposition \ref{simplicial-metric} as well and combine it with the pulled back semimetric by the
period map. Several metrics are constructed in this way in \S 7 and \S 8.
In \S 9, we also view $d_0$ as an analogue
of the Weil-Petersson metric of the Teichm\"uller space
and follow the  construction of  the McMullen metric
on Teichm\"uller spaces \cite{mc} in order to construct another invariant
complete geodesic metric on $X_n$. 

This brief discussion explains the fruitful analogy between the three
kind of spaces: symmetric spaces $X$, the Teichm\"uller space $\T_{g,n}$, 
and the outer space $X_n$.

\section{Tropical curves}

Metric graphs are closely related to tropical curves.
We will define tropical curves and show that 
there is a 1-1 correspondence between
the set of compact metric graphs without vertex of valence 1 or 2
and the set of compact {\em smooth} tropical curves.

This relation with tropical curves allows us to explain
why the definition of the positive
definite quadratic form $Q$  in Equation \ref{quad-form}
for a metric graph is natural.
We also point out in Remark \ref{trop-moduli} that outer spaces can be
used to study the moduli space $\M_n^{trop}$ of tropical curves. 

To make this paper more self-contained and easier for the reader,
we have tried to explain some results in tropical geometry starting from the basics. 
One reason that tropical curves are confusing to some people, in particular
to the author of this paper, is that one often uses tropical plane curves as examples of
tropical curves. But these curves are {\em singular in general} and do not correspond  bijectively
to metric graphs,
and only {\em smooth} tropical curves correspond to metric graphs.
Definitions of abstract smooth tropical curves (and manifolds) basically
follow the usual definition of manifolds in the sense that transition functions between
local charts
are given by $\Z$-affine transformations, but there are several
different ways to describe local models and they can be
complicated to people who approach tropical curves for the first time.
For example, the local model in \cite[\S 3]{mikz} is not geometric. 
In \cite[Definition 1.15]{ims}, a tropical curve is defined as a connected metric graph.
We will follow the approach in \cite{mik3} together with some ideas and statements in \cite{mik1} \cite{mik2}. There is another approach in \cite{rst}. 

\subsection{The tropical semifield and tropical polynomials}

Briefly, the tropical geometry is algebraic geometry over the tropical semifield, and
the tropical semifield is $\bf T=\R \cup \{-\infty\}$ with suitable operations.
 
It has an addition $``x+y"=\max\{x, y\}$, and $-\infty$ is the additive zero.
It has a multiplication $``x\cdot y"=x+y$, and 
$0$ is the multiplicative identity. Note that $-\infty$ is also
the multiplicative $0$, since $``x\cdot (-\infty)"=-\infty$.

But it does not have a subtraction, since we cannot
recover $x$ from $``x+y"=\max\{x, y\}$ and $y$ when $x< y$.

The division $``x/y"=x-y$ is defined when $y\neq -\infty$.
This corresponds to the usual rule that the division is defined when the denominator
is not zero.

 Once we have addition and multiplication, we can define polynomials.
 For example, for a polynomial in two variables $p(x, y)=``ax^2+bxy+cy^2"$,
 when it is restricted to $\R^2\subset \bf T^2$, it is the maximum of linear functions
 $$p(x, y)=\max\{2x+a, x+y+b, 2y+c\}.$$

 \subsection{Tropical plane curves}
 
 It is well-known that for any polynomial $p(x, y)$, the equation
 $p(x, y)=0$ defines a plane curve in $\R^2$ (or rather in $\C^2$).
 Similarly, tropical polynomials  define tropical curves in $\bf T^2$.
 We will concentrate on their points inside $\R^2$.
 
 Given a tropical polynomial $p(x, y)$, a direct generalization is to
 define the associated tropical curve by $p(x, y)=-\infty$. 
 (Recall that $-\infty$ is the tropical additive 0).  
 But this has a serious problem of having too few points.
 For example, for a degree 2 polynomial $p(x, y)$ above, $p(x, y)=-\infty$
 is equivalent to $x=-\infty, y=-\infty$.

 For a usual polynomial $p(x, y)$ over $\C$, the plane curve $p(x, y)=0$ 
 in $\C^2$ can also be defined to consist of points $(x, y)$
 where the function $1/p(x,y)$ is not regular.
 
 For a tropical polynomial, we can similarly define {\em its tropical plane curve}
 in $\R^2$
 to consist of points where the quotient $``0/p(x, y)"=-p(x, y)$ is not regular.
 (Recall that 0 is the tropical multiplicative identity.) See \cite[pp. 10-11]{ims} for more details.
 
 This means that points of the tropical plane curve of $p(x, y)$
 correspond to points where $p(x, y)$ is achieved by at least two linear
 functions.  (Note that near a point where two linear functions are equal,
 $-p(x,y)$ is not equal to another tropical polynomial, since $-\max\{f, g\}=\inf\{-f, -g\}$).

 The following result is clear from the definition.
 
 \begin{prop}
 For any tropical polynomial $p(x, y)$,  its tropical plane  curve in $\R^2$ 
 is a piecewise linear curve with rational slopes.
\end{prop}

Though the slope of each line segment is rational,  its endpoints are not
necessarily rational. For some pictures of plane tropical curves, see \cite{mik1} \cite{mik2}. 
Similarly, a tropical polynomial $p(x, y, z)$ defines a tropical hyperplane
in $\mathbf T^3$, and two tropical polynomials $p_1(x, y, z), p_2(x, y, z)$ define 
a tropical subvariety  in $\mathbf T^3$ as the intersection of the tropical hyperplanes.

Tropical plane curves are not smooth tropical curves whenever it contains a vertex of valence
strictly greater than 3.
In order to relate metric graphs to tropical curves, we need to introduce {\em smooth} tropical
curves which are not necessarily embedded in some ambient tropical space $\mathbf T^n$.
In fact, compact tropical curves cannot be embedded in $\mathbf T^n$ (which can be seen
by using the balancing condition at vertices),
and tropical curves in $\mathbf T^n$ are affine tropical curves. 

\subsection{Abstract smooth tropical curves}

In this subsection, we introduce the notion of an abstract smooth tropical curve.
The basic reference is \cite{mik3}, and the papers \cite{ims} \cite{ss} are also helpful. 

To distinguish tropical curves from graphs, we use $\ga$ to denote tropical curves.
{\em A smooth tropical curve} is a graph $\ga$ endowed with a {\em 
smooth $\Z$-affine structure},  which consists of an open covering by compatible
$\Z$-affine smooth charts:
\begin{enumerate}
\item For each interior point $x$ of an edge $e$, let $U$ be a neighborhood of $x$
homeomorphic to an interval. A $\Z$-affine structure on $U$ is an embedding
$$\phi_U: U\to \R.$$

\item If $x$ is a vertex of valence $n+1$, let $e_0, \cdots, e_n$ be the edges
connected to $x$. Let $U$ be a neighborhood of $x$ which is contained in the union
of these edges and which consists of $n+1$ intervals. 
A {\em smooth $\Z$-affine structure} on $U$ is an embedding 
$$\phi_U: U\to \R^n$$
such that for every edge $e_i$, $\phi_U(U\cap e_i)$ is a line segment
with a rational slope, and $\phi_U$ satisfies a {\em balancing and nondegeneracy condition}
at the vertex $x$.
Specifically, let $v_i\in \Z^n$ be the {\em primitive vector} in the direction of $\phi_U(U\cap e_i)$
(pointing away from $x$).
Then the {\em balancing condition} at $x$ is the equation
\begin{equation}\label{balance-0}
\sum_{i=0}^{n}v_i=0;
\end{equation}
and the {\em nondegeneracy condition}
 is that any $n$ vectors of the $n+1$ vectors $v_0, \cdots, v_n$ form a basis of $\Z^n$,
  which is  equivalent to the condition that $v_0, \cdots, v_{n-1}$ form a basis of $\Z^n$
when the balancing condition is satisfied.
(One reason for imposing this non-degeneracy condition is that every two bases of $\Z^n$
can be interchanged by some element of $\mathrm{GL}(n, \Z)$).

\item The above local charts are compatible as follows.
Given two overlapping charts $\phi_1: U_1\to \R^{n_1},$ $\phi_2: U_2\to \R^{n_2}$,
we can choose a common $\R^N$ and inclusions of $\R^{n_1}, \R^{n_2} \subset \R^N$ as
linear subspaces 
$$\phi_1: U_1 \to \R^{n_1}\subset \R^N, \phi_2: U_2\to \R^{n_2}\subset \R^N,$$
and find a $\Z$-affine linear map $\Phi_{12}: \R^{N}\to \R^{N}$
such that

$$\phi_1|_{U_1\cap U_2}=\Phi_{12} \circ \phi_2.$$
(Recall that a {\em $\Z$-affine linear map} is given by
$v\in \R^{N}\mapsto Av+b$, where $A$ is an invertible integral $N\times N$-matrix, 
i.e., $A\in \mathrm{GL}(n, \Z)$, $b\in \R^{N}$.)
Note that the overlap $U_1\cap U_2$ can be complicated if $U_1, U_2$ are large.
If they are sufficiently small, then $U_1\cap U_2$ consists of an open interval since
$\ga$ is a graph. 
\end{enumerate} 

The relation between tropical plane curves and smooth tropical curves is explained
in the next proposition.

\begin{prop}\label{normalization}
Given a tropical plane curve $C$, there is a canonical smooth tropical curve $J$
which can be  mapped onto the tropical plane curve $C$ by a locally $\Z$-affine map. 
The smooth tropical curve $J$ is called the normalization of the tropical plane curve $C$.
\end{prop}

Suppose a tropical curve $C\subset \R^2$ is defined by a tropical polynomial $p(x, y)$. 
Then $C$ is a piecewise-linear curve in the plane such that each edge has a rational slope.
When $C$ contains a vertex $x_0$ of valence strictly greater than 3, 
the primitive integral vectors along all edges of the vertex $x_0$ do not satisfy the
nondegeneracy condition.

Now we explain how to put a smooth $\Z$-affine structure on $C$. 
Suppose $x_0\in C$ is a vertex of valence $n+1$. Let $e_0, \cdots, e_n$ be all the
edges from $x_0$.
Suppose an edge $e$ is defined by the equality of two linear functions
$``a_1 x^{j_1} y^{k_1}"=``a_2 x^{j_2} y^{k_2}"$, i.e.,
$j_1 x+ k_1 y+ a_1= j_2 x+ k_2 y+ a_2$.
Define $V(e)$ to be the {\em primitive integral vector} in the direction of the edge $e$
(pointing away from the vertex),
and the {\em weight} $W(e)$ of the edge to be the greatest common denominator 
$GCD(j_1-j_2, k_1-k_2)$.

\begin{lem}[{\cite[p. 512]{mik1}}]\label{balance}
Under the above notation, for every vertex $x_0$ of a plane tropical curve 
$C$, the following balancing condition is satisfied:
\begin{equation}\label{balance-weight}
\sum_{i=0}^n W(e_i) V(e_i)=0.
\end{equation}
\end{lem}
\begin{proof}
Suppose an edge $e$ out of $x_0$ is defined by $f_1(x, y)=f_2(x, y)$,
where $f_1(x, y)=j_1 x+ k_1 y+ a_1$, and $f_2(x, y)= j_2 x+ k_2 y+ a_2$.
Assume that when we move counter-clockwise around the vertex, 
$f_1(x, y)-f_2(x, y)$
changes from positive to negative.
 Then it can be checked that $W(e)V(e)=(-(k_1-k_2), j_1-j_2)=
(k_2-k_1, j_1-j_2)$. On the other hand, if $f_1(x, y)-f_2(x, y)$
changes from negative  to positive, then $W(e)V(e)=(k_1-k_2, j_2-j_1)$. (Think of the
simplest example of two linear functions $x, y$ with the vertex at the origin.)
For any vertex $x$ of $C$, there are linear functions $f_1(x, y), \cdots, f_k(x, y)$
such that all edges are defined by the equality of two linear functions
which achieve the maximum value. 
These edges divide a small disk with center at $x$ into cone pieces, and 
for each cone $\Omega$, one linear function $f_i$ is the unique function which takes the maximum value on
 this cone. When we move counter-clockwise across the two edges $e_1, e_2$ of $\Omega$ 
 defined by $f_i=f_{i_1}$ and 
 $f_i=f_{i_2}$ for some $f_{i_1}$ and $f_{i_2}$,  
 both possibilities of sign changes of the values $f_i-f_{i_1}$ and 
 $f_i-f_{i_2}$ can occur. This means the coefficient of $x$
 (and similarly the coefficient of $y$) of $f_i$ will appear in $W(e_1)V(e_1)$
 and  $W(e_2)V(e_2)$ once as positive and another time as negative. By summing over all
 edges (or rather all cones), they all cancel out and the sum
  $\sum_{i=0}^n W(e_i)V(e_i)$ is zero. (Think of the simplest example 
 of tropical curve defined by the tropical polynomial $``x+y+0"$, which has only one
 vertex at the origin and 3 edges coming out of it.)
\end{proof} 

\begin{cor}
For every tropical polynomial $p(x, y)$, its associated tropical plane curve $C$
is a piecewise linear curve in $\R^2$ with rational slopes such that with respect to
a suitable weight $W(e)$ for each edge $e$, it satisfies the balancing condition as
in Equation \ref{balance-weight}
at every vertex.
\end{cor}

\begin{rem}
{\em
The converse of the above corollary is also true \cite[p. 512]{mik1}. 
Let $C$ be a piecewise linear curve in $\R^2$ with rational slopes such that with respect to
a suitable integral weight for each edge it satisfies the balancing condition in Equation \ref{balance-weight}
at every vertex.
Then there exists a tropical polynomial whose associated tropical curve is $C$,
and the weight of every edge is defined as above. Briefly, this can be proved as in 
Lemma \ref{balance} by noting that if a finite sequence of numbers $a_1, \cdots, a_k$ sums to 0,
then there exist numbers $b_1, \cdots, b_k$ such that $a_i=b_i-b_{i+1}$ for $i=1, \cdots, k$.
}
\end{rem}

Consider the following piecewise linear curve $\ga_0$
 in $\R^n$ with one vertex of valence $n+1$
at the origin: the edges out of the origin are rays along the $n+1$ vectors:
$v_0=(1, \cdots, 1), v_1=(-1, 0, \cdots, 0), \cdots,$ $ v_n=(0, \cdots, 0, -1)$.
When $n=2$, this is the plane tropical curve defined by the polynomial $p(x, y)=``x+y+0"$.

Clearly the edges have rational slopes, and the primitive vectors along the edges
are $v_0, \cdots, v_n$. 
These integral vectors satisfy the balancing and nondegeneracy conditions
and hence  $\ga_0$ is a smooth tropical curve.

\begin{prop}Suppose $x$ is a vertex of a tropical plane curve $C$ of valence $n+1$.
Let $e_0, \cdots, e_n$ be the edges out of $x$. For each edge $e_i$,
let $W(e_i)$ be its weight
and  $V(e_i)$ be its primitive integral vector as defined above. 
Then there exists a unique $\Z$-affine linear map $\pi: \R^n \to \R^2$ such that
$$ \pi(0)=x, \quad \pi(v_i)=W(e_i)V(e_i), \quad i=0, \cdots, n.$$
\end{prop}
\begin{proof}
The nondegeneracy condition of $v_0, \cdots, v_n$ implies that $v_1, \cdots, v_n$ are linearly
independent vectors of $\R^n$ and  form a basis of $\R^n$. 
This implies that there exists a unique affine linear map $\pi: \R^n \to \R^2$ such that
$\pi(v_i)=W(e_i)V(e_i)$, $i=1, \cdots, n$, and $\pi(0)=x.$
Since the vectors $v_0, \cdots, v_n$ and the edges $e_0, \cdots, e_n$ satisfy the balancing
condition, i.e., $\sum_{i=0}^n v_i=0$ and $\sum_{i=0}^n W(e_i)V(e_i)$,
it follows that
$\pi(v_0)=W(e_0)V(e_0)$. 
Since  $v_1, \cdots, v_n$ forms  a basis of the lattice $\Z^n$
and $W(e_1)V(e_1), \cdots, W(e_n)V(e_n)$ are integral vectors,
the map $\pi$ is a $\Z$-affine linear map.
\end{proof}

Let $U_0$ be a small neighborhood of the origin in $\ga_0$.
Then when $U_0$ is small enough, $\pi(U_0)$ is a neighborhood of $x$, denoted by $U$,
and the restricted map $\pi_{U_0}: U_0\to U$ is a homeomorphism.
Its inverse  gives a smooth $\Z$-affine structure on $U$:
$$\phi_U=(\pi_{U_0})^{-1}: U\to U_0\subset \R^n.$$

For each point $x\in C$ contained in the interior
of an edge $e$, let $U$ be a neighborhood of $x$ contained in $e$.
Let $V(e)$ be the primitive integral vector of $e$, and $W(e)$ its weight.
Let $v$ be  the unit vector of $\R$ from 0 to 1. Define a $\Z$-affine linear map
$\pi: \R\to \R^2$ by
$$\pi(0)=x, \quad \pi(v)=W(e)V(e).$$ 
Then for a suitable neighborhood $U_0$ of $0$ in $\R$,
$\pi: U_0\to U$ is a homeomorphism, where $U=\pi(U_0)$.
The inverse $\phi_U=\pi^{-1}: U\to U_0\subset \R$
defines a $\Z$-affine structure on $U$. (We note that if we treat $x$
as a vertex of valence 2, this is a special case of the  above general construction
for vertices.)

\begin{prop}\label{normailzation-2} 
For any tropical plane curve $C$, the $\Z$-affine structures on neighborhoods
of points of $C$ defined above give the structure of a smooth tropical curve
on the graph associated with $C$. Denote this smooth tropical curve by $J$.
\end{prop}
\begin{proof}
We need to show that the local charts are compatible.
Let $U_1, U_2$ be two charts of $C$. By our choice,
$U_1\cap U_2$ is an open interval of an edge $e$ of $C$.
Under the embeddings $\phi_1: U_1\to \R^{n_1}, \phi_2: U_2\to \R^{n_2}$,
the integral vector $W(e)V(e)$ along the edge $e$ is mapped to
primitive integral vectors  in $\R^{n_1}, \R^{n_2}$. Embed $\R^{n_1}, \R^{n_2}$
into a common linear space $\R^N$. Since every two primitive integral vectors of $\R^N$ 
are related by an element $g\in \mathrm{GL}(n, \Z)$,  they are compatible. 
\end{proof}
 
 \noindent{\em Proof of Proposition \ref{normalization}.}

By definition, the underlying space of the tropical curve $J$ in Proposition \ref{normailzation-2} 
is the same as $C$, and the identity map from $J$ to $C$ is locally a $\Z$-affine map.
Roughly, the normalization involves both making the edge vectors of every vertex nondenegerate
and scaling $\Z$-affine maps on the edges according to their weights.  

 \begin{rem}
 {\em
 From the above discussion, it is clear that the integral structure $\Z^n$ in $\R^n$ is used crucially. 
 If all the weights $W(e)$ are equal to 1, then primitive integral vectors of edges
 are mapped to primitive integral vectors and the normalization amounts to making
the primitive integral vectors of edges nondegenerate by embedding them into a 
bigger-dimensional linear space.
 }
 \end{rem}
 
\subsection{Identification between smooth tropical curves and metric graphs}

From the definition, it is clear that a smooth tropical curve $\ga$ is topologically a graph.
The following results are well-known and often taken as definitions \cite[Proposition 1.3]{mik3}
\cite[Proposition 3.6]{mikz} \cite[Definition 1.15]{ims}.
We fill in more details for the convenience of the reader. 
 
\begin{prop}\label{tro-metric}
Every smooth tropical curve $\ga$ admits a canonical metric and hence becomes a metric graph.
\end{prop}
\begin{proof}
We will define a metric on a smooth tropical curve $\ga$ using a $\Z$-affine structure on charts.
It suffices to define lengths of subintervals of the edges. 
Let $\phi_U: U\to \R$ be a $\Z$-affine chart of an interior point of an edge $e$.
Let $V$ be an open subinterval contained in $U$. Define the length of $V$ to be equal to
the length of the image $\phi_U(V)$ with respect to the standard metric of $\R$. 

If $x$ is a vertex of valence $n+1$ and $\phi_U: U\to \R^n$ is a $\Z$-affine chart,
and $e_0, \cdots, e_n$ are the edges from $x$, and $V(e_i)$ an integral primitive
vector on the ray $\phi_U(U\cap e_i)$, 
we restrict the standard metric of $\R^n$
to $\phi_U(U\cap e_i)$ and scale the metric on
it so that the length of the vector $V(e_i)$ is equal to 1
(i.e., scale the metric on {\em each edge separately}
so that the length of a primitive integral vector along it is equal to 1). 
Then for an interval $V$ of an edge $e_i$ contained in $U$,
the length of $V$ is equal to the length of the image $\phi_U(V)$ with respect to the
scaled metric on the ray  $\phi_U(U\cap e_i)$.

In summary, if a directed
interval of a smooth tropical curve $\ga$ is mapped to a primitive integral
vector under a local $\Z$-affine chart, its length is equal to 1. 
Since these $\Z$-affine structures on overlapping charts are compatible,
integral primitive vectors in one chart are also integral primitive vectors
in another, and the metrics of these charts agree on the overlaps.
By gluing these metrics together, we obtain a metric on the tropical curve $\ga$.
\end{proof}

The converse is also true.

\begin{prop}\label{metric-tro}
Every metric graph admits a canonical structure of a smooth tropical curve. 
\end{prop}
\begin{proof}
Let $(G, \ell)$ be a metric graph, and $x$ is a vertex of valence $n+1$.
As in the previous subsection,
let $\ga_0$ be the tropical curve in $\R^n$ with a vertex at $0$ of valence $n+1$,
and 
the edges out of the origin are rays along the vectors
$v_0=(1, \cdots, 1), v_1=(-1, 0, \cdots, 0), \cdots, v_n=(0, \cdots, 0, -1)$.
Put a metric on $\ga_0$ as in the previous proposition.
Map a small neighborhood $U$ of $x$ {\em isometrically} to a small neighborhood
$U_0$ of $0$. This defines a $\Z$-affine structure on $U$. 

If $x\in G$ is an interior point of an edge, and $U$ is an interval of the edge
and contains $x$, we can map $U$ {\em isometrically} to a neighborhood of the origin
of $\R$. This defines a $\Z$-affine structure on $U$.

For every large integer $N$, an interval of length $\frac{1}{N}$ of $\ga$ is mapped
to $\frac{1}{N}$-multiple of a primitive integral vector in $\R^n$, where $\R^n$ depends on the
local chart. Using this, 
it can be checked that these $\Z$-affine structures on charts
on $G$ are compatible and hence define a structure of smooth tropical
curve on the metric graph $(G, \ell)$.  
\end{proof}

The basic point of the proof of the above proposition is that a $\Z$-affine structure
on a one-dimensional manifold is equivalent to a length metric. 
Because of this proposition, in some places such as \cite[Definition 1.15]{ims},
a tropical curve (or rather a smooth tropical curve)
is defined as a metric graph. 

\begin{rem}
{\em We can show easily that the two maps in Proposition \ref{metric-tro}
 are the inverse of each other. 
Specifically,
  start with a smooth tropical curve $J$,  obtain a metric graph $(G, \ell)$  by Proposition 
\ref{tro-metric}, and then obtain a new tropical curve $J'$ from $(G, \ell)$
by Proposition \ref{metric-tro}. It can be shown that $J'$ is isomorphic to $J$, and
the other way of starting with a metric graph can be checked as well.
}
\end{rem}

\begin{rem}
{\em If we start from a tropical plane curve $C\subset \R^n$, we can directly
put a metric on it as follows \cite[p. 512]{mik1}. For every linear segment $e$ of $C$, let $V(e)$
be a primitive integral vector on $e$, and $W(e)$ be the weight of $e$ defined above.
The metric on $e$ is the scaling of the restriction of the standard metric of $\R^2$
 such that the norm of $V(e)$ is equal to $1/W(e)$, in particular, the norm of $V(e)$
is equal to 1 if and only if the weight $W(e)=1$.
We can check that this metric graph is the same metric graph associated with the
normalization $J$ of $C$, which is a smooth tropical curve.

This also allows one to construct alternatively a normalization of  the tropical plane curve
$C$, i.e., by taking the tropical curve
corresponding to the metric graph associated with $C$. 
}
\end{rem}

\section{Jacobian varieties of  Riemann surfaces and the Jacobian map}

Several results on tropical curves have been motivated by results on algebraic curves over
$\C$,
or Riemann surfaces. A particularly important one for us is the Jacobian variety
of a compact Riemann surface. In this section, we recall in detail the notion of
polarized abelian varieties to motivate the corresponding tropical Jacobian
variety for a compact
smooth tropical curve. We also introduce the Jacobian map and 
we state the Torelli Theorem
for Riemann surfaces.

For a compact Riemann surface $\Sigma_g$, we can define its Jacobian variety as follows.
Let $H^{0,1}(\Sigma_g)=H^0(\Sigma_g, \Omega)$ be the space of holomorphic 1-forms on $\Sigma_g$.
It is a complex vector space of dimension $g$. 
The first homology group $H_1(\Sigma_g, \Z)$ is isomorphic to $\Z^{2g}$
and it can be embedded into the dual space $H^0(\Sigma_g, \Omega)^*$ 
of $H^0(\Sigma_g, \Omega)$, $\pi: H_1(\Sigma_g, \Z) \to H^0(\Sigma_g, \Omega)^*$, 
by integration:
For every 1-cycle $\sigma$ and a 1-form $\omega\in H^0(\Sigma_g, \Omega)$,
$$\pi(\sigma)(\omega)=\int_{\sigma}\omega.$$

The complex torus $H_1(\Sigma_g, \Z) \backslash H^0(\Sigma_g, \Omega)^*$ admits a
{\em canonical polarization}
 coming from the pairing $H_1(\Sigma_g, \Z)\times H_1(\Sigma_g, \Z)
\to \Z$ and is a {\em principally polarized Abelian variety}.  It is called the {\em Jacobian variety} of $\Sigma_g$. 
 To explain this,
we discuss the notion of a polarization of an Abelian variety. See \cite{grh} for more details.

Recall that a line bundle $L\to M$ over an algebraic variety $M$ is called {\em very
ample} if its global sections, i.e., elements of $H^0(M, \mathcal O(L))$,
give an embedding $M\to \mathbb P^N$, or equivalently,  if there exists an
embedding $\pi: M\to   \mathbb P^N$, and $L=\pi^* H$, where $H$ is the line bundle
over $\mathbb P^N$ corresponding to a hyperplane of $\mathbb P^N$.
A line bundle $L\to M$ over an algebraic variety $M$ is called {\em ample}
if  for $k\gg 0$, $L^k\to M$ is very ample.
Such an ample line bundle $L\to M$ is called a {\em polarization} of $M$.
The existence of an ample line bundle over $M$ is a necessary and sufficient 
condition for $M$ to be a projective variety.

The Chern class of an ample line bundle $L\to M$ is an integral class in $H^2(M, \Z)$
and can be represented by a closed, positive $(1,1)$-form in the De Rham  cohomology
$H^2(M, \R)$.
Conversely, by the Kodaira embedding theorem, every such {\em positive integral class}
in $H^2(M, \R)$ is the Chern class of an ample line bundle over $M$.

Therefore, a {\em polarization} of an algebraic variety $M$ corresponds to
a {\em positive integral class} in $H^2(M, \R)$. A $(1,1)$-form representing such a
class is called a {\em Hodge form}. 

A complex torus $M=\Lambda\backslash \C^n$ is {\em an  Abelian variety} if it
is a projective variety, i.e., if it admits a polarization.
The existence of a Hodge form  on $\Lambda\backslash \C^n$
depends on the compatibility of the complex structure of $\C^n$ 
and the integral structure of $\C^n$ defined by the lattice $\Lambda$.

Given a basis $\lambda_1, \cdots, \lambda_{2n}$ of $\Lambda$ and 
 a $\C$-basis $e_1, \cdots, e_n$ of $\C^n$, define the {\em period matrix}  $\Omega=
 (\omega_{\alpha i})$
 of size $n\times 2n$ by
 $$\lambda_i=\sum_\alpha \omega_{\alpha i} e_\alpha.$$
 
 By using the fact that every cohomology class can be represented by a
 differential form  which is invariant under translation, one can obtain
 the following {\em Riemann condition for Abelian varieties} \cite[p. 306]{grh}:
A complex torus $M=\Lambda\backslash \C^n$ is an Abelian variety
if and only if there exists a basis $\lambda_1, \cdots, \lambda_{2n}$ of $\Lambda$
and a $\C$-basis $e_1, \cdots, e_n$ of $\C^n$ such that the period matrix
$\Omega$ is of the form
$$\Omega=(\Delta_\delta, Z),$$
where $\Delta_\delta$ is a diagonal matrix with positive integral
 diagonal entries $\delta_1, \cdots, \delta_n$
such that $\delta_1| \delta_2,  \cdots, \delta_{n-1} | \delta_n$,
and $Z$ is symmetric and Im $Z$ is positive definite.

More precisely, let $x_1, \cdots, x_{2n}$ be the coordinates of $\C^n$ determined by 
$\lambda_1, \cdots, \lambda_{2n}$. 
Then under the above condition on the period matrix, the integral 2-form
\begin{equation}\label{hodge-form}
\omega=\sum \delta_\alpha dx_\alpha \wedge d x_{n+\alpha}
\end{equation}
is positive and hence is a Hodge form, which gives a polarization of the complex torus
$\Lambda\backslash \C^n$.
If $\delta_1=\cdots =\delta_n=1$, the polarization is called {\em principal},
and the Abelian variety is called {\em principally polarized}.

Conversely, suppose $\omega$ is a Hodge form which is invariant under translation. 
Then there exists a basis
$\lambda_1, \cdots, \lambda_{2n}$ of $\Lambda$ such that
$\omega$ is of the form in Equation \ref{hodge-form}.
Now computing the form $\omega$ in the complex coordinates determined by $e_1,
\cdots, e_n$ shows that the condition that $\omega$ is a positive (1,1)-form
is exactly equivalent to the condition on the period matrix $Z$.

When $M=H_1(\Sigma_g, \Z)\backslash H^0(\Sigma_g, \Omega)^*$,
for every symplectic basis $a_1, b_1, \cdots, a_g, b_g$ of $H_1(\Sigma_g, \Z)$,
there exists a normalized basis $\omega_1, \cdots, \omega_g$
 of $H^0(\Sigma_g, \Omega)$  such that $\int_{a_i} \omega_j=\delta_{ij}$
 \cite[p. 231]{grh}.
(Note that it is not automatic that for every $j$
there exists a holomorphic differential form
$\omega_j$ such that the period conditions $\int_{a_i} \omega_j=\delta_{ij}$, $i=1, \cdots, g$,
 are satisfied. Once this is known, it is easy to see that different $\omega_j$
are linearly independent forms.)

This implies that with respect to the dual basis of $H^0(\Sigma_g, \Omega)^*$,
the period matrix $(\Delta_\delta, Z)$ of the symplectic basis $a_1, b_1, \cdots, a_g, b_g$ is principal,
i.e., $\Delta_\delta=I_g$. The Riemann bilinear relation \cite[p. 232]{grh}
implies that the period $Z$ is symmetric and that Im $Z$ is positive definite.    
Therefore, the Riemann condition is satisfied by the Jacobian variety
 $H_1(\Sigma_g, \Z)\backslash H^0(\Sigma_g, \Omega)^*$, and the Jacobian variety
is a {\em principally polarized Abelian variety}. 

\begin{rem}\label{char-polar}
{\em The Hodge form $\omega$ is a $(1,1)$-form in 
$H^2(H_1(\Sigma_g, \Z)\backslash H^0(\Sigma_g, \Omega)^*)$,
and hence corresponds to {\em an integral skew-symmetric bilinear
form} on $H_1(\Sigma_g, \Z)$.
It turns out to be
 the intersection form on $H_1(\Sigma_g, \Z)$:
$$H_1(\Sigma_g, \Z)\times H_1(\Sigma_g, \Z) \to \Z,$$
which is a nondegenerate skew-symmetric bilinear form by the Poincare duality.

We will see later that  the Chern class of a line bundle over a tropical curve lies
in the first cohomology group of the tropical curve
and that a polarization corresponds to a positive definite symmetric bilinear
form instead of a skew-symmetric bilinear form. One reason is that a Riemann surface is of
real dimension 2 and a tropical curve is of real dimension 1.
}
\end{rem}

Let $\mathfrak h_g=\{X+i Y\mid X, Y \text{\ are real symmetric $g\times g$-matrices}, Y>0\}$
be the {\em Siegel upper half space}. 
Then the symplectic group $ \mathrm{Sp}(2g, \R)$ acts holomorphically
and transitively on $\mathfrak h_g$, and $\mathfrak h_g\cong  \mathrm{Sp}(2g, \R)/\mathrm{U}(g)$,
a Hermitian symmetric space of noncompact type. 

The above discussion shows that
for every compact Riemann surface $\Sigma_g$, its period $Z$ is in $\mathfrak h_g$.
Different choices of symplectic bases of $H_1(\Sigma_g, \Z)$ lead to different points
in a $ \mathrm{Sp}(2g, \Z)$-orbit in $\mathfrak h_g$. This gives a well-defined map
\begin{equation}
\Pi: \M_g \to  \mathrm{Sp}(2g, \Z)\backslash \mathfrak h_g, \quad \Sigma_g\mapsto 
 \mathrm{Sp}(2g, \Z)Z.
\end{equation}
This is called the {\em period map} for Riemann surfaces.

Let $\mathcal A_g$ denotes the moduli space of {\em principally polarized Abelian
varieties} of dimension $g$. Then the above discussion shows that $\mathcal A_g
\cong    \mathrm{Sp}(2g, \Z)\backslash \mathfrak h_g$, where a point $\tau\in \mathfrak h_g$
corresponds to a principally polarized abelian variety $\Z^n +\tau \Z^n \backslash \C^n$,
and the period map is equivalent to the 
{\em Jacobian map}:
\begin{equation}\label{torelli}
\Pi: \M_g \to \mathcal A_g, \quad \Sigma_g \mapsto  H_1(\Sigma_g, \Z)\backslash H^0(\Sigma_g, \Omega)^*.
\end{equation}

The Jacobian map is a complex analytic map with respect to the natural
complex structures  on $\M_g$ and $ \mathrm{Sp}(2g, \Z)\backslash \mathfrak h_g$. 
It is also a morphism when $\M_g$ and $\mathcal A_g$ are
given the structure of algebraic varieties.
The  {\em Torelli Theorem} \cite[p. 359]{grh} says that $\Pi$ {\em is an embedding}. 
The Jacobian map and the Torelli theorem have played a fundamental role
in understanding the moduli space $\M_g$.

\begin{rem}
{\em The complex torus  $H_1(\Sigma_g, \Z) \backslash H_1(\Sigma_g, \R)$
is usually called the {\em Albanese variety} of $\Sigma_g$. Using the duality between
$H_1(\Sigma, \R)$ and $H^1(\Sigma_g, \R)$ (and hence also $H^{0,1}(\Sigma_g)$),
it can be seen that it is isomorphic to the Jacobian variety of $\Sigma_g$.
In general, the Albanese variety of a higher dimensional complex manifold
or algebraic variety is a generalization of the Jacobian variety of a projective
curve (or a compact Riemann surface). 
}
\end{rem}

\section{Jacobians of tropical curves and metric curves}

As mentioned before, metric graphs considered in this paper have no vertices
of valence 1 or 2.
By the results in Propositions \ref{tro-metric} and \ref{metric-tro}, metric graphs can be identified
with smooth tropical curves. From now on, we will denote metric graphs by $(\ga, \ell)$
instead of $(G, \ell)$ to be consistent with the notation for tropical curves.

Recall from Equation \ref{quad-form} and Proposition \ref{pos-definite}
that for any compact metric graph $(\ga, \ell)$, 
there is a positive definite quadratic form $Q$ on
the real vector space $C_1(\ga, \R)$ of 1-chains 
on $\ga$ with $\R$-coefficients, 
which defines an inner product 
on $H_1(\ga, \R)$ and descends to a flat metric on the torus $H_1(\ga, \Z)\backslash
H_1(\ga, \R)$.

Following the definitions in \cite{ks} \cite{cvi} for (combinatorial) graphs, we have the following
definition. 

\begin{defi}
The torus $H_1(\ga, \Z)\backslash
H_1(\ga, \R)$ with the flat metric induced from the quadratic form $Q$ is called
the Albanese torus of the metric graph $(\ga, \ell)$, and denoted by
$(H_1(\ga, \Z)\backslash H_1(\ga, \R), Q)$.
\end{defi}

\begin{rem}
{\em The definition of the Albanese and Jacobian tori for graphs
were introduced in \cite[p. 94]{ks}.
For a graph $G$, the torus $H_1(G, \Z)\backslash H_1(G, \R)$ with the flat metric
where each edge is assigned length 1 is the Albanese torus,
and the dual flat torus $H^1(G, \Z)\backslash H^1(X, \R)$ is called the Jacobian torus.
These are motivated by, but slightly different from, 
various notions of Jacobian varieties of algebraic
varieties, Kahler manifolds and Riemannian manifolds, 
which are defined as quotients of cohomology
groups.
As recalled in the previous section, 
for a Riemann surface $\Sigma_g$, the Jacobian variety is a quotient of the dual 
of the first cohomology, $H_1(\Sigma_g, \Z)\backslash H^0(\Sigma_g, \Omega)^*$,
where $H^0(\Sigma_g, \Omega)=H^{0,1}(\Sigma_g)\subset H^1(\Sigma, \C)$.
}
\end{rem}

In this metric on the space $C_1(\ga, \R)$ of 1-chains (or the
metric of the torus $H_1(\ga, \Z)\backslash
H_1(\ga, \R)$), the length of an edge $e$ as a 1-chain 
is equal to $\sqrt{\ell(E)}$. 
As mentioned before, this makes the definition of the metric $Q$  seem unnatural 
since its dependence on the metric $\ell$ of the graph $(\ga, \ell)$ is not smooth.
It might be more tempting
to define $Q$ so that  the length of the 1-chain $e$ to be
$\ell(E)$  in order for it to depend
smoothly on the lengths of edges of the graph $(\ga, \ell)$.
(Of course, for graphs whose edges have normalized length equal to 1,
there is no difference in the above two choices.)

In order to explain that the definition  of the quadratic form in Equation \ref{quad-form}
is the right one and that $Q$ gives a {\em principal polarization} when
$H_1(\ga, \Z)\backslash H_1(\ga, \R)$ is given a {\em natural structure of a tropical torus},
 we will identify 
$(H_1(\ga, \Z)\backslash H_1(\ga, \R), Q)$
 with the tropical Jacobian variety of the tropical curve associated with $(\ga, \ell)$. 
(We note that for all metric graphs $(\ga, \ell)$ of genus $g$,  
$H_1(\ga, \R)$ are isomorphic, and $H_1(\ga, \Z)$ are isomorphic lattices.
As recalled below, a tropical structure on $H_1(\ga, \R)$ is determined by another
lattice.)

First, we need to define differential forms on a tropical curve,
and the polarization of a tropical torus.
Then we show that the tropical Jacobian variety of a compact smooth tropical curve
has a canonical principal polarization. We will follow \cite{mikz}. 

Given a compact smooth tropical curve $J$, a {\em 1-form}  $\omega$ 
(i.e., a tropical differential form of degree 1) on $J$ is a collection of {\em 
constant real differential 1-forms}
on the edges of $J$
satisfying the balancing condition at every vertex of $J$.
Specifically, let $x$ be a vertex of valence $n+1$,
and let $U$ be a chart of $x$ with a $\Z$-affine structure
$\phi_U : U\to \R^n$. 
Let $e_0, \cdots, e_n$ be the edges out of the vertex $\phi_U(x)$ in $\R^n$,
and $V(e_i)$ a primitive integral vector along $e_i$ pointing away from the vertex
 $\phi_U(x)$. The 1-form $\omega$ defines a constant differential 1-form
on each $e_i$ as a smooth differential manifold,
and the {\em balancing condition} is
\begin{equation}\label{balance1}
\sum_{i=0}^n \omega(V(e_i))=0.
\end{equation}
It is clear that this balancing condition is preserved by $\Z$-linear maps and hence the
balancing condition is well-defined.

Let $\Omega(J)$ be the space of global 1-forms on $J$. This is a finite dimensional
vector space.

\begin{lem}[{\cite[\S 6.1]{mikz}}]\label{global-form}
If the genus of a compact tropical smooth
curve $J$ is equal to $g$, then $\Omega(J)$ is a real vector space
of dimension $g$.
\end{lem}
\begin{proof}
Since $J$ is homotopy equivalent to  a wedge of $n$ circles, there are $g$ points,  $p_1, \cdots,
p_g$,  contained in 
the interior of edges of $C$ such that $C-\{p_1, \cdots, p_g\}$ is contractible.
They are called breakpoints.
From the balancing conditions at the vertices, it is clear that a global 
1-form on $J$ is determined by its values on primitive integral tangent vectors
of the edges containing these breakpoints $p_1, \cdots,
p_g$. It can also be seen that  all values at the breakpoints
can be achieved by some 1-forms. For example, this is true for the wedge of $n$-circles.
By noting that the balancing conditions are preserved by contracting edges of the complement
of the breakpoints, we can prove this by induction on the number of edges by starting
from the wedge of circles. 
This proves that $\Omega(J)$ is of dimension $g$.
\end{proof}

\begin{defi}
A 1-form $\omega \in \Omega(J)$ is called  integral if it takes an integral value
on every integral tangent vector of every edge of $J$.
\end{defi}

We note that the notion of integral tangent vectors of each edge of $J$ can be
defined using $\Z$-affine local charts and it is well-defined. 
By the same proof of Lemma \ref{global-form}, we obtain the following.

\begin{lem}
Let $\Omega_\Z(J)$ denote the subspace of integral 1-forms on $J$.
Then $\Omega_\Z(J)$ is a lattice of the real vector space $\Omega(J)$.
\end{lem}

Let $\Omega(J)^*$ be the dual space of linear functionals on $\Omega(J)$.
We now define a pairing between $H_1(J, \Z)$ and $\Omega(J)$,
and hence a map
$$p: H_1(J, \Z) \to \Omega(J)^*.$$
(Note that $p$ stands for period.)
For any path $\gamma: [a, b]\to J$, the pull back $\gamma^* \omega$ is a piecewise-constant 
differential 1-form on $[a, b]$
 (with possible discontinuity at the preimage of the vertices),
and the integral $\int_\gamma \omega$ is defined to be $\int_a^b \gamma^* \omega$.
This implies that for any cycle $\sigma$ and any 1-form $\omega$,
we can define
\begin{equation}\label{period-map-2} 
p(\sigma)(\omega)=\int_\sigma \omega.
\end{equation}

Clearly, this map $p: H_1(J, \Z) \to \Omega(J)^*$ is linear and can be extended to
a linear transformation 
$$p: H_1(J, \R) \to \Omega(J)^*.$$

\begin{prop}\label{lattice-embed}
The period map $p: H_1(J, \Z) \to \Omega(J)^*$ defined above is an embedding,
and the linear transformation $p: H_1(J, \R) \to \Omega(J)^*$ is an isomorphism. 
\end{prop}
\begin{proof}
If we identify a tropical curve $\ga$  with its associated metric graph $(\ga, \ell)$, 
and on each edge $e$ use we the length
function $t$ measured from a vertex as the parameter, then the restriction of a 1-form
$\omega$ to $e$ is given by $a\ dt$, where $a$ is a constant.
At each vertex $x$ of the graph $\ga$, let $e_1, \cdots, e_k$ be the edges out of $x$,
and $t_1, \cdots, t_k$ be the length functions measured from $x$.
Suppose that $\omega$ is a 1-form, and on each edge $e_i$, $\omega=a_i\ d t_i$.
Then {\em the balancing condition} at $x$ for the 1-form $\omega$ in Equation (\ref{balance1})
 is equivalent to 
\begin{equation}\label{balance-2}
\sum_{i=1}^k a_i=0.
\end{equation}

If we choose an orientation of an edge $e$, then there is a choice of a length parameter $t_e$
such that $\int_e dt_e>0$.
Fix an orientation of all edges $e$ of $\ga$ and the compatible choices of the length
parameters $t_e$. 
For any nontrivial cycle $\sigma=\sum_{e} a_e a$, define a 1-form
$\omega=\sum_{e} a_e\ dt_e$ on $\ga$, 
which means that it restricts to $a_e d t_e$ on each edge $e$.
The compatibility of the edges contained in $\sigma$ at the vertices
to make it a cycle implies that
$\omega$ is a 1-form on $\ga$.
Since
$$p(\sigma)(\omega)=\sum_{e} a_e^2\int_e\ dt_e=\sum_{e} a_e^2 \ell(e)>0,$$
this implies that $p(\sigma)$ is nonzero, and  hence $p$ is an embedding.

Since the dimensions of $H_1(J, \R)$ and $\Omega(J)^*$ are both
equal to $n$, the extended linear transformation is an isomorphism,
and the proposition is proved. 
\end{proof}

\begin{rem}\label{cycle-form-1}
{\em The above interpretation of 1-forms $\omega$ on a tropical curve $\ga$  in terms of
length parameters of a metric graph $(\ga, \ell)$ gives a different proof of Lemma \ref{global-form}.
Fix an orientation of edges $e$ of $\ga$ and the corresponding choices of the length
parameters $t_e$ as above.
We can define a linear isomorphism between $H_1(\ga, \R)$ and $\Omega(\ga)$.
For every 1-cycle $\sigma=\sum_{e} a_e e$, define a 1-form
$\omega=\sum_{e} a_e \ d t_e$, which means that the restriction of $\omega$ to each edge
$e$ is equal to $a_e$. Conversely, for every 1-form
$\omega=\sum_{e} a_e \ d t_e$, we have a 1-cycle $\sigma=\sum_{e} a_e e$. 
Therefore, the dimension of  $\Omega(\ga)$ is equal to the dimension of $H_1(\ga, \R)$,
which is equal to $g$. 
}
\end{rem}

By Proposition \ref{lattice-embed}, 
the quadratic form $Q$ on $H_1(J, \R)$ defines a quadratic form on $\Omega(J)^*$.
The positive definite quadratic form
$Q$ on $\Omega(J)^*$ gives an isomorphism $\Omega(J)^*\to \Omega(J)$, and together with
the map $p: H_1(J, \Z)\to \Omega(J)^*$ in Proposition  \ref{lattice-embed}, it
gives 
an embedding 
\begin{equation}\label{int-cycle}
i_Q: H_1(J, \Z)\to \Omega(J).
\end{equation}

 We emphasize that this embedding $i_Q$ depends on the positive definite quadratic form $Q$.
The following proposition explains the {\em reason for
the choice of value} $Q(e, e)$ for an edge $e$
in Definition (\ref{quad-form}) of the quadratic form $Q$ of a compact
metric graph.

\begin{prop}\label{principal}
In the above notation, the image of the embedding $i_Q: H_1(J, \Z)\to \Omega(J)$ is
equal to the space of $\Omega_\Z(J)$ of integral 1-forms. 
\end{prop}
\begin{proof} We need to show that for any 1-cycle $\sigma=\sum_{e}a_e e$ in $H_1(J, \R)$, 
the image $i_Q(\sigma)\in \Omega_\Z(J)$ if and only if $\sigma\in H_1(J, \Z)$.
For every edge $e$ of $\ga$, 
let $V(e)$ be an integral primitive vector tangent to $e$. 
By the definition of the metric on the tropical curve,
the norm of $V(e)$ is equal to 1. 
By definition of the integral 1-forms $\omega$, $i_Q(\sigma)$ is integral
if and only if $i_Q(\sigma)(V(e))$ is integral for every edge $E$, which is equivalent to 
the fact that 
$\int_e i_Q(\sigma)$ is an integral multiple of $\ell(e)$.
On the other hand, by definition,
$$\int_e i_Q(\sigma)=Q(e, \sigma)=a_e \ell(e).$$ 
Therefore, $i_Q(\sigma)$ is an integral 1-form if and only
if $\sigma$ is an integral 1-cycle, i.e., all coefficients $a_e$ are in $\Z$.  
\end{proof}

 \begin{rem}\label{cycle-form-2}
 {\em If we orient all edges $e$ of $\ga$ and pick  compatible length
 parameters $t_e$ for them as in Proposition \ref{lattice-embed}, then 
 a 1-form $\omega$ on $\ga$ is integral if and only if for every edge $e$,
 the restriction of $\omega$ to $e$ is an integral multiple of $d t_e$,
 and the map
 in Proposition \ref{principal} can be written down explicitly as follows:
 For every cycle $\sigma=\sum_{e} a_e e$, the restriction of the 1-form
 $i_Q(\sigma)$
 to the edge $e$ is equal to  $a_e d t_e,$ i.e., $i_Q(\sigma)=\sum_{e} a_e \ d t_e$.
 As explained in the proof of Proposition \ref{principal},
 the definition of the quadratic form $Q$ on $H_1(\ga, \Z)$ was motivated
 to obtain this natural equivalence. 
 
 Specifically, $Q$ is the unique positive definite quadratic form
which gives an isomorphism $\Omega(\ga)\cong \Omega(\ga)^*$
such that  once it is composed with the isomorphism 
$H_1(\ga, \R)\to \Omega(\ga)$ in Remark \ref{cycle-form-1},
 the map $H_1(\ga, \Z)\to \Omega(\ga)^*$ is equal to the period
 map $p: H_1(\ga, \Z)\to \Omega(\ga)^*$ in Equation (\ref{period-map-2}). 
 Since both the period map and the isomorphism in Remark \ref{cycle-form-1} are natural,
 the definition of $Q$ and the induced isomorphism  $\Omega(\ga)\cong \Omega(\ga)^*$
are also natural.
 }
 \end{rem}

 Before continuing,  we  recall the definition of  tropical torus according to 
\cite[\S 5.1]{mikz}.
{\em A tropical torus is a real torus $\Lambda \backslash \R^n$ with a $\Z$-affine structure. }
A {\em $\Z$-affine structure} on $\R^n$ is determined by an {\em integral structure}, 
i.e., a lattice $L\subset \R^n$. 
Unless specified otherwise, the $\Z$-affine structure is determined
by the standard lattice $\Z^n$ in $\R^n$.
Then for any lattice $\Lambda\subset \R^n$, the torus 
$X=\Lambda\backslash \R^n$ has the induced
$\Z$-affine structure.  (Note that $\Lambda$ is not commensurable with $\Z^n$ 
or does not satisfy other relations with $\Z^n$ in general). 

More generally, suppose $V$ is a real vector space and $L\subset V$ is a lattice.
Then $V$ has an integral and $\Z$-affine structure determined by $L$.
For any lattice $\Lambda\subset V$, the torus $\Lambda\backslash V$ is a tropical
torus, whose tropical structure is determined by the lattice $L$.

The subspace  $\Omega_\Z(J)$ of integral 1-forms defines a dual
 lattice $\Omega_\Z(J)^*$ in
$\Omega(J)^*$,  and hence defines a $\Z$-affine structure on $\Omega(J)^*$
and also on the torus $H_1(J, \Z) \backslash \Omega(J)^*$.

\begin{defi}
For a smooth tropical curve $J$,  
the torus  $H_1(J, \Z) \backslash \Omega(J)^*$ with the $\Z$-affine structure
defined above by the space  $\Omega_\Z(J)$ of integral 1-forms
is called the tropical Jacobian variety of the tropical curve $J$.
\end{defi}

\begin{rem}
{\em
Clearly, $H_1(J, \R)$ can be mapped isomorphically to $\Omega(J)^*$ by integration
as in Proposition \ref{lattice-embed},
and $H_1(J, \Z) \backslash \Omega(J)^*$ is isomorphic to the Albanese torus
$H_1(J, \Z) \backslash H_1(J, \R)$. Under this isomorphism, the $\Z$-affine structure
on $H_1(J, \Z) \backslash H_1(J, \R)$, or rather on
$H_1(J, \R)$, depends on the positive definite quadratic form $Q$ 
or the metric $\ell$ of the graph (see Proposition \ref{tropical-flat}).
In general, it is not the $\Z$-affine structure
determined by the integral lattice $H_1(J, \Z)$ in $H_1(J, \R)$. 
}
\end{rem}

Besides defining a flat metric on $H_1(J, \Z) \backslash \Omega(J)^*$,
the positive definite quadratic form $Q$ also defines a polarization, which is defined
as follows. 
Since transition functions between trivializations of a tropical line bundle
over different charts of $X$ are affine linear functions with integral slope, a
tropical line bundle over $X$ corresponds to an element $H^1(X, \mathcal O^*)$,
where $\mathcal O^*=\mathrm{Aff}$ is the sheaf of affine linear functions with integral slope.

Recall that the sheaf  $\mathbb T^*_\Z$ of locally constant integral 1-forms
is defined by the exact sequence

$$0\to \R\to \mathcal O^*\to  \mathbb T^*_\Z\to 0.$$
We note that  the sheaf  $\mathbb T^*_\Z$ is isomorphic to the locally constant
sheaf associated with the abelian group $(\Z^n)^*$. 
This induces a map

$$c: H^1(X, \mathcal O^*)\to H^1(X, \mathbb T^*_\Z).$$
This is called the {\em Chern class} map. 

In the case of a complex torus, an ample line bundle $L$ is called a {\em polarization},
or rather its first Chern class is called a polarization. 

In the tropical case, we call {\em the Chern class} $c(L)$ of a tropical line bundle $L$
also {\em a polarization of a tropical torus}, according to See \cite[\S 5.1]{mikz}. It seems more natural to
require also that $c(L)$ is positive in the sense defined below.

Now we  unwrap the definitions and identify the Chern  class of a tropical
line bundle of a tropical torus $X$ with a quadratic form. 
Since $X=\Lambda\backslash \R^n$, 
$$H^1(X, \mathbb T^*_\Z)\cong \Lambda^*\otimes (\Z^n)^*\cong \mathrm{Hom}(\Lambda,
(\Z^n)^*).$$ 
Then a polarization $c(L)$ of a tropical torus $X$
 is a linear map $c(L): \Lambda\to (\Z^n)^*$,
which is equivalent to a bilinear form $Q(L):\Lambda \times \Z^n \to \Z$,
which in turn induces a bilinear form $Q(L)$ on $\R^n$,
$Q(L): (\Lambda \otimes \R) \times (\Z^n\otimes \R)\to \Z\otimes \R$.

\begin{rem}
{\em 
We note that for a complex torus $\Lambda\backslash \C^n$
to admit a polarization, its complex structure of $\C^n$ should be compatible with
the integral structure imposed on it by $\Lambda$.
Similarly, for a real torus $\Lambda\backslash \R^n$ to be a tropical torus
admitting a polarization $c(L)$, the $\Z$-affine structure of $\R^n$ imposed
by the lattice $\Z^n$ should be compatible with the lattice $\Lambda$
in the sense that $c(L)$ induces a linear map $c(L):  \Lambda\to (\Z^n)^*$. 
}
\end{rem}

\begin{prop}[{\cite[\S 5.1]{mikz}}]\label{symmetric-ploar}
For every tropical line bundle $L$ on $X$, the induced bilinear form
$Q(L)$ on $\R^n$ is a symmetric bilinear form, and every symmetric bilinear
symmetric form $Q$ on $\R^n$ that is integral 
in the sense that it comes from a linear map $\Lambda\to (\Z^n)^*$
is the Chern class of a tropical line bundle over $\Lambda\backslash \R^n$.
\end{prop}
\begin{proof}
In the long exact sequence,
$$\to H^1(X, \mathcal O^*)\to H^1(X, \mathbb T^*_\Z) \to H^2(X, \R),$$
if we identify $H^1(X, \mathbb T^*_\Z)\cong \Lambda^*\otimes (\Z^n)^*$
and $H^2(X, \R)\cong \wedge^2(\R^n)^*$, then one can prove that
the map $H^1(X, \mathbb T^*_\Z) \to H^2(X, \R)$ is the restriction
of the skew-symmetrization map $(\R^n)^*\otimes (\R^n)^*
\to \wedge^2(\R^n)^*.$ Then the proposition is clear. 
\end{proof}

If the associated symmetric bilinear form $Q(L)$ of a tropical line bundle $L$ 
is {\em positive}, the tropical torus $X$ with the line bundle $L$
is called {\em a polarized tropical Abelian variety}.  The index of the image $c(L)(\Lambda)$
in $(\Z^n)^*$ is called {\em the index of the polarization}. If the index is equal to 1,
then the polarization $L$ is called {\em a principal polarization}.

Naturally, a {\em tropical abelian variety} is defined to be
a tropical torus that admits a tropical line bundle
whose quadratic form is positive definite. 

\begin{prop}\label{polarized}
The tropical torus $H_1(J, \Z) \backslash \Omega(J)^*$ with the positive
definite quadratic form $Q$
is a principally polarized tropical abelian
variety. 
\end{prop}
\begin{proof}
The quadratic form $Q$ on $\Omega(J)^*$ is positive definite
and maps $H_1(J, \Z)$ isomorphically to $(\Omega_\Z(J)^*)^*=\Omega_\Z(J)$ 
by Proposition \ref{principal}.
Therefore, by Proposition \ref{symmetric-ploar}, it comes
from a tropical line bundle over the tropical torus $H_1(J, \Z)\backslash \Omega(J)^*$, 
and the induced polarization on $H_1(J, \Z)\backslash \Omega(J)^*$ is principal.
\end{proof}

\begin{rem}
{\em
For a polarized abelian variety $M$, the Hodge class $\omega$ defines a non-degenerate
skew-symmetric bilinear form on $H_1(M, \Z)$ and hence a linear map $i_\omega: 
H_1(M, \Z)\to H_1(M, \Z)^*$. It is principal if and only if the image $i_\omega(H_1(M, \Z))$
has index 1 in $H_1(M, \Z)^*$. 
See Remark \ref{char-polar} for the case of Jacobian varieties.
Therefore, the definition of a principal polarization of a tropical abelian
variety is consistent with the definition of a principal polarization 
for complex abelian varieties. 
}
\end{rem}

For applications of  tropical Jacobian varieties to metric graphs,
we need to identify principally polarized tropical tori with compact flat Riemmannian tori.

\begin{prop}\label{tropical-flat}
There exists a natural 1-1 correspondence between the space $\mathcal A_n^{trop}$
of principally polarized tropical abelian varieties of dimension $n$
and the space $\mathcal {FT}_n$ of compact flat Riemannian tori of dimension $n$.
\end{prop}
\begin{proof} For every  principally polarized tropical abelian variety $(\Lambda\backslash \R^n, Q)$,
the positive definite quadratic form $Q$ defines a flat Riemannian metric on $\R^n$
and hence also on the compact torus $\Lambda\backslash \R^n$. 
Though the  integral structure $\Z^n$ (or $\Z$-affine
structure) on $\R^n$ and $\Lambda\backslash \R^n$ does not show
up explicitly in the compact flat torus $\Lambda\backslash \R^n$,
it is uniquely determined by the condition that the 
positive definite quadratic form  $Q$ 
has to map $\Lambda$ isomorphically to $(\Z^n)^*$, $\pi_Q: \Lambda\to (\Z^n)^*$.
In other words, the lattice $\Z^n$ in $\R^n$ is uniquely determined by the pair $(Q, \Lambda)$.

Conversely, for any flat compact Riemannian torus $\Lambda\backslash \R^n$, 
we claim that
there is a {\em canonical $\Z$-affine structure}
on $\R^n$ such that $\Lambda\backslash \R^n$ becomes a principally polarized tropical abelian variety.
Let $Q$ be the quadratic form on $\R^n$
which defines the flat metric of $\Lambda\backslash \R^n$.
This defines a linear isomorphism $\pi_Q: \R^n \to (\R^n)^*$ 
which maps $\Lambda$ to a lattice $L^*$ in $(\R^n)^*$, 
which induces a dual lattice in $L=(L^*)^*$
 in $\R^n$.
The lattice $L$ defines a $\Z$-affine structure on $\R^n$
and hence a $\Z$-affine structure on 
$\Lambda\backslash \R^n$, 
turning $\Lambda\backslash \R^n$ into a 
tropical torus.
It can be checked directly that the symmetric form $Q$ is integral 
in the sense of Proposition \ref{symmetric-ploar}
 and hence defines a polarization on the tropical torus
$\Lambda\backslash \R^n$, and it is also a principal polarization
by the choice of the $\Z$-affine structure. 
\end{proof}

To determine the moduli space $\mathcal A_n^{trop}$, we need to identify the space
$\mathcal {FT}_n$.

\begin{prop}
The moduli space $\mathcal {FT}_n$ of compact flat Riemannian
 tori of dimension $n$ can be identified with the locally symmetric space
$\mathrm{GL}(n, \Z)\backslash \mathrm{GL}(n, \R)/\mathrm{O}(n)$. 
\end{prop}
\begin{proof}
Let $Q_0$ be the standard positive quadratic form on $\R^n$
which induces the standard inner product $\langle \cdot, \cdot \rangle_0$ on $\R^n$.
Then every compact flat torus is isometric to $(\Lambda\backslash \R^n, Q_0)$
for some lattices.
A basic point of the geometry of numbers is that every flat torus
$(\Lambda\backslash \R^n, Q_0)$ is isometric to $(\Z^n\backslash \R^n, Q)$ for some 
positive definite quadratic form. In fact, if $A$ is the matrix formed by the column vectors
which represent a basis of $\Lambda$, then $Q=A^T A$.
It is clear that two flat tori $(\Z^n\backslash \R^n, Q_1)$ and $(\Z^n\backslash \R^n, Q_2)$
are isometric if and only if their quadratic forms $Q_1$ and $Q_2$ are equivalent under 
$\mathrm{GL}(n, \Z)$.
Since the space of equivalent classes of positive quadratic forms on $\R^n$
is equal to $\mathrm{GL}(n, \Z)\backslash \mathrm{GL}(n, \R)/\mathrm{O}(n)$, 
the proposition is proved. 
\end{proof}

The combination of the above two propositions gives the following corollary, which
was stated as \cite[Assumption  2, p. 165]{cvi} (see also \cite[\S 6.4]{mikz}).

\begin{cor}\label{moduli-abelian}
The moduli space $\mathcal A_n^{trop}$ of principally polarized tropical abelian varieties
of dimension $n$ is equal to the locally symmetric space $\mathrm{GL}(n, \Z)\backslash \mathrm{GL}(n, \R)/\mathrm{O}(n)$:
$$\mathcal A_n^{trop} \cong \mathrm{GL}(n, \Z)\backslash \mathrm{GL}(n, \R)/\mathrm{O}(n),$$
and $\dim \mathcal A_n^{trop} =\frac{n(n+1)}{2}.$
\end{cor}

\begin{rem}\label{trop-moduli}
{\em
Another assumption made in \cite[Assumption 1, p. 164]{cvi} concerns the existence
as a tropical variety  
of the moduli space $\M_n^{trop}$ of compact smooth tropical curves of genus $n$.
More discussions are given in \cite{ca}, where $\M_n^{trop}$ was constructed as a
topological space and its connection with the outer space $X_n$ was mentioned \cite[\S 5]{ca}.
A more concrete outline was given in \cite[\S 3.1]{mik2}.
Certainly, the construction of $\M_n^{trop}$ as a tropical 
variety is an important problem in tropical geometry.
We provide some details and references
to show that outer spaces of metric graphs can be important for this
purpose. This might also contribute to the analogy with Teichm\"uller spaces of Riemann surfaces
(Teichm\"uller spaces were introduced in order to study the moduli spaces of Riemann
surfaces, for example, to endow the latter with a complex structure)
and shows that outer space theory can pay back to the tropical geometry in contrast
to the results discussed so far on applications of tropical geometry to metric graphs.

As it is known, the moduli space of Riemann surfaces $\M_g$ is an algebraic variety
but not a complex manifold, on the other hand, the Teichm\"uller space $\T_g$ is a complex manifold but not an
algebraic variety. 
Let $\hat{X}_n$ be the {\em unreduced} outer space of {\em unnormalized}
marked metric graphs $(\ga, \ell, h)$
of genus $n$, where by the unreduced outer space, we mean that the graphs
 can contain separating edges, and by unnormalized, we mean 
that the total sum of edge lengths, $\sum_e \ell(e)$,
is not required to be 1. 
For each marked graph $(\ga, h)$, there is a 
simplicial polyhedral cone $\hat{\Sigma}_{(G, h)}\subset
\R^k_{\geq0}$.
When $G$ does not contain any separating edge,
 the homothety section of the simplicial polyhedral cone $\hat{\Sigma}_{(G, h)}$
 is the simplex $\Sigma_{(G,h)}$ in the outer space 
$X_n$ defined above in \S 2.  

It is clear that $\hat{X}_n$ is a simplicial polyhedral complex and $\mathrm{Out}(F_n)$ acts 
on $\hat{X}_n$ by simplicial maps which are also $\Z$-affine maps, 
and the quotient $\mathrm{Out}(F_n)\backslash \hat{X}_n$
by $\mathrm{Out}(F_n)$ is equal to the moduli space $\M_n^{trop}$.

The assertion is that {\em $\hat{X}_n$ is a tropical space}, or rather is a {\em ``tropical manifold" (a space
which locally has the structure of a tropical variety)},  but not
a tropical variety,
and  that $\M_n^{trop}$ is a {\em tropical orbifold.}

According to the definition of tropical varieties \cite{mik2}, locally,  $\hat{X}_n$ should be embedded into
$\R^N$ as a polyhedral complex  with integer slope which satisfies the balancing condition
at every polyhedral face of co-dimension 1. 
It is clear that each polyhedral cone of $\hat{X}_n$ has integer slope, but 
for every polyhedral cone $F$ of positive codimension, it is not obvious how to embed
all polyhedral cones of $\hat{X}_n$ containing the face $F$ into some $\R^N$
 as a polyhedral complex with integer slopes
satisfying the balancing condition.
The idea \cite[\S 3.1]{mik2} is to note that a marked metric graph
$(G, \ell, h)$ in the interior of $F$ has vertices $v$ of valence at least 4.
To move to points of the polyhedral cones which contain $F$ as a face
(i.e., to move to points in a neighborhood of $(G, \ell, h)$ in $\hat{X}_n$),
we need to expand such vertices $v$ into trees with marked boundary points,
and to vary the lengths of edges of $G$ as well. 
Now trees with marked boundary points correspond to marked rational tropical curves,
and their moduli space $\M_{0, k}$ turns out to be a tropical variety  
\cite[Theorem 3.7]{gkm} \cite{mik4}. Since varying the edge lengths of $G$ clearly
traces out a neighborhood of a tropical variety, and since the product of tropical
varieties is a tropical variety, 
this implies that a neighborhood of the marked metric graph $(G, \ell, h)$ 
in $\hat{X}_n$ can be embedded into some $\R^N$ as a polyhedral
complex with integer slope satisfying the balancing condition at every face
of codimension 1. Therefore, $\hat{X}_n$ is a ``tropical manifold". 
}
\end{rem}

\section{The Torelli Theorem for tropical curves}

Let $\M_n^t$ be the moduli space of compact smooth tropical curves of genus $n$.
For every tropical curve $J\in \M_n^t$, by Proposition \ref{polarized},
its tropical Jacobian variety
$(H_1(J, \Z)\backslash \Omega(J)^*, Q)$ is a principally polarized
tropical abelian variety.
By Corollary \ref{moduli-abelian}, we obtain 
the tropical Jacobian map:
\begin{equation}\label{tropical-jacobian}
\Pi^{trop}: \M_n^t \to \mathcal A_n^{trop}=\mathrm{GL}(n, \Z)\backslash 
\mathrm{GL}(n, \R)/\mathrm{O}(n), \quad J\mapsto
(H_1(J, \Z)\backslash \Omega(J)^*, Q).
\end{equation}

By Proposition \ref{lattice-embed}, 
$(H_1(J, \Z)\backslash \Omega(J)^*, Q)
 \cong (H_1(J, \Z)\backslash H_1(J, \R), Q)$,
and we also denote the Jacobian of $J$ 
by $(H_1(J, \Z)\backslash H_1(J, \R), Q)$.

\begin{rem}
{\em To understand better the analogy with the Jacobian (or period) map of Riemann
surfaces, we can interpret the quadratic form $Q$ of a  compact tropical curve $J$
 as the period of some
normalized 1-forms on $J$. Choose a basis  $\sigma_1, \cdots, \sigma_g$
of $H_1(\ga, \Z)\cong \Z^g$.
For each 1-cycle $\sigma_i$, let $\omega_i$ be the
corresponding 1-form as in Remark \ref{cycle-form-1}. 
Then the period of $\omega_j$ on the cycle $\sigma_i$
is $\int_{\sigma_i}\omega_j=Q(\sigma_i, \sigma_j)$,
i.e., the matrix of the quadratic form $Q$ with respect to the basis $\sigma_1, 
\cdots, \sigma_g$ is the period matrix of the forms $\omega_1, \cdots, \omega_g$.
}
\end{rem}

As mentioned in \S 4, the Jacobian map $\Pi: \M_g \to \mathcal A_g$
for compact Riemann surfaces is an embedding, 
but the tropical Jacobian map  $\Pi^{trop}$ is not injective.  
This is reasonable since the Jacobian variety of a tropical
curve $J$ only depends on the flat metric $Q$ of the torus $H_1(J, \Z)\backslash H_1(J, \R)$,
which depends only on the lengths of 1-cycles in $J$, but as a metric graph $(J, \ell)$,
$J$ depends on the lengths $\ell(e)$ of all edges in $J$. So some information is lost
in passing to the Jacobian variety $(H_1(J, \Z)\backslash H_1(J, \R), Q)$. More precise
statements are contained in Proposition \ref{3-edge}  below.

Now we summarize several results from \cite{cvi} on the exact extent of  failure of the tropical
Jacobian map $\Pi^{trop}$. 

\begin{defi}
Two metric graphs $(\ga, \ell)$, $\ga', \ell')$ are called cyclically equivalent
if there exists a bijection between their edges $\varepsilon: E(\ga) \to E(\ga')$
such that $\ell(e)=\ell'(\varepsilon(e))$, and $\varepsilon$
induces a bijection between the
cycles of $\ga$ and the cycles of $\ga'$. (A cycle is a subgraph that is
connected and homotopy equivalent to a
circle, but does not contain separating edges.) 
\end{defi}

\begin{defi} For $k\geq 2$, 
a graph $\ga$ having at least 2 vertices is said to be $k$-edge connected if the
graph obtained from $\ga$ by removing  any $k-1$ edges is connected. 
\end{defi}

For example, a graph is 2-edge connected if and only if it is connected and
does not contain any separating edge.

To obtain highly edge-connected graphs, we introduce two types of edge contractions:
\begin{description}
\item (A) contract a separating edge,
\item (B) contract one of two edges which form a separating pair of edges $\{e_1, e_2\}$,
i.e., $\ga-\{ e_1\},  \ga-\{e_2\}$ are connected, but $\ga-\{e_1, e_2\}$ is disconnected.
\end{description}

\begin{defi}
The 2-edge connectivization of a connected graph $\ga$ is the 2-edge connected graph
$\ga^2$ which is obtained from $\ga$ by iterating operation (A), i.e., contracting separating edges.
A 3-edge connectivization of a connected graph $\ga$ is a 3-edge connected graph
$\ga^3$ which is obtained from the 2-edge connectivization $\ga^2$ of $\ga$ 
by iterating operation (B). 
\end{defi}

Note that there is no unique 3-edge connectivization of a connected graph $\ga$ in general,
but all 3-edge connectivizations are cyclically equivalent \cite[Lemma 2.3.8]{cvi}, and they retain
some information of the original graph.

For this purpose, we need the notion of $C1$-sets of $\ga$
which corresponds to edges of $\ga^3$. 
For every subset $S\subset E(\ga)$ of edges of $\ga$, define a subgraph $\ga -S$ which is obtained
by removing the edges in $S$ but leaving the vertices unchanged.
A new, complementary graph $\ga(S)$ is obtained from $\ga$ by contracting every
connected component of $\ga-S$ to a single point. (Different connected components are
contracted to different points.)

\begin{defi}[{\cite[Definition 2.3.1]{cvi}}]\label{C1-set}
Let $\ga$ be a connected graph that does not contain any separating edge.
A subset $S\subset E(\ga)$ is called a C1-set of $\ga$ if $\ga(S)$ is a cycle
and $\ga-S$ does not contain any separating edge. Denote the collection
of C1-sets of $\ga$ by
$\mathrm{Set}^1 (\ga)$. 
\end{defi}

We discuss some examples of C1-sets. If $S=\{e\}$ and $\ga-\{e\}$ does not contain
any separating edge, then $S$ is a C1-set.
For example, suppose that $\ga$ has two vertices connected by more than 3 edges.
If $S$ contains more than one edges, then $S$ is not a C1-set since $\ga(S)$ 
is homotopy equivalent to a wedge of more than one circles. 
In this case, every edge $E$ is a C1-set.
If $\ga$ contains two vertices and two edges, then each edge $E$ is not a C1-set
since $\ga-S$ contains a separating edge, which is, in fact, the only edge of the graph $\ga-S$.
More generally, if $\ga$ is homeomorphic to a circle (i.e., all vertices are of valence 2),
then the only C1-set $S$ is the set of all edges due to the second condition that $\ga-S$
does not contain any separating edge. 

For a general graph $\ga$, if $S=\{e_1, e_2\}$ is a separating pair of edges, 
then $S$ is a C1-set if and only if $\ga-S$ does not contain any separating edge. 
The following general fact is true.

\begin{lem}[{\cite[Corollary 2.3.4]{cvi}}]
A graph $\ga$ is a 3-edge connected graph if and only if it is connected and
there is a bijection between $E(\ga)$ and $\mathrm{Set}^1(\ga)$, i.e.,
every edge is a C1-set and every C1-set consists of exactly one edge.  
\end{lem}

\begin{prop}[{\cite[Lemma 2.3.8]{cvi}}]
Let $\ga$ be a graph, and $\ga^3$ be a 3-edge connectivization of $\ga$. Then the following
holds:
\begin{enumerate}
 \item The genus of $\ga^3$ is equal to the genus of $\ga$.
\item There is a canonical bijection between the following three sets:
$\mathrm{Set}^1(\ga^3)$, $E(\ga^3)$,
and $\mathrm{Set}^1(\ga)$. 
\item All 3-edge connectivizations of $\ga$ are cyclically equivalent.
\end{enumerate}
\end{prop}

We can define similar notions for metric graphs.

\begin{defi}[{\cite[Definition 4.1.7]{cvi}}]
A 3-edge connectivization of a metric graph $(\ga, \ell)$ is a metric graph $(\ga^3, \ell^3)$,
where $\ga^3$ is a 3-edge connectivization of $\ga$ and $\ell^3$ is the length function
defined by
$$\ell^3(e_S)=\sum_{e\in S}\ell(e),$$
where for each C1-set $S$ of $\ga$, $e_S$ is the corresponding edge of $\ga^3$
under the bijection between $\mathrm{Set}^1(\ga)$ and $E(\ga^3)$ in the above proposition.
\end{defi}

It follows from the above definition of the metric and the above proposition that
all 3-edge connectivizations of a metric graph are cyclically equivalent as metric graphs.

The Torelli Theorem for metric graphs can now be stated.

\begin{prop}[{\cite[Theorem 4.1.10]{cvi}}]\label{3-edge}  
Two metric graphs (or equivalently  tropical curves) $\ga$ and $\ga'$ 
have isomorphic tropical Jacobian varieties
if and only if their 3-edge connectivizations $\ga^3$ and $(\ga')^3$ are cyclically
equivalent. 
\end{prop}

\begin{cor}\label{failure}
When $n\geq 3$,  the tropical Jacobian map 
$$\Pi^{trop}: 
\M_n^t \to \mathcal A_n^{trop}=\mathrm{GL}(n, \Z)\backslash \mathrm{GL}(n, \R)/\mathrm{O}(n)$$ is not injective.
\end{cor}
\begin{proof}
There are several reasons for the failure of the injectivity of $\Pi^{trop}$. 
If $\ga$ contains separating edges, then all separating edges
are contracted in $\ga^3$. But as metric graphs, the lengths of these separating edges matter.
For the outer space of metric graphs, in this paper, we mean the reduced outer space
consisting of marked metric graphs with no separating edges (or bridges). 
This does not cause a problem.  

Even for graphs which do not contain any separating edges, the tropical 
Jacobian map can still fail to be injective.
Take a graph $\ga$ such that some C1-set $S$ of $\ga$ contains at least
2 edges, $e_1, \cdots, e_k$, $k\geq 2$. 
Now in the 3-edge connectivization $\ga^3$, the edge $e_S$ corresponding to
the set $S$ has length $\ell^3(e_S)=\ell(e_1)+\cdots +\ell(e_k)$.
By Proposition \ref{3-edge},  
the Jacobian variety of $(\ga, \ell)$ is the same as the Jacobian variety
of its 3-edge connectivization $(\ga^3, \ell^3)$, and it  depends 
on the sums $\sum_{e\in S}\ell(E)$ for C1-sets $S$.
On the other hand, as metric graphs, they depend on the lengths of edges $\ell(e)$, $e\in S$.
There are certainly infinitely many nonisomorphic metric graphs with the same value
of $\ell^3(e_S)=\ell(e_1)+\cdots +\ell(e_k)$. 

Some examples of C1-sets that contain more than one element are explained
after Definition \ref{C1-set} (see also Remark \ref{more-exam} below). 
In particular, we see that for every $g\geq 3$, 
there exist metric graphs of genus $g$ which do not contain
any separating edge but contain C1-sets $S$ with cardinality of
at least 2, for example, separating pairs of edges. This proves Corollary \ref{failure}. 
\end{proof}

\begin{rem}\label{more-exam}
{\em To construct a graph $\ga$ which contains a large C1-set $S$,
we start with a graph $\ga_1$ which is homeomorphic to the circle and has
$k\geq 3$ edges. 
But the valence of every vertex of $\ga_1$ is equal to 2. 
To overcome this, we attach a connected graph whose vertices have valence of at least 3
at every vertex of $\ga_1$. Denote the new graph by $\ga$.
Take $S$ to consist of the edges of $\ga_1$.
Then it can be seen that $S$ is a C1-set of $\ga$ with $k$ edges.
This explains that there are sets of metric graphs of arbitrarily large dimension 
which are mapped to a point under the tropical Jacobian map 
$\Pi^{trop}$.
}
\end{rem}

\section{Tropical Jacobian map and invariant complete geodesic metrics on the outer space}

As discussed in the previous section, the tropical Jacobian variety of tropical curves defines
the tropical Jacobian map in Equation (\ref{tropical-jacobian}):
$$\Pi^{trop}: \M_g^t\to \mathcal A_n^{trop}\cong \mathrm{GL}(n, \Z)\backslash \mathrm{GL}(n, \R)/\mathrm{O}(n).$$
By Propositions \ref{tro-metric} and \ref{metric-tro}, there is a natural map
$X_n\to \M_g^t$, which factors through the action of $\mathrm{Out}(F_n)$
and gives an embedding
$$\mathrm{Out}(F_n)\backslash X_n \to \M_g^t$$ and hence  a {\em tropical Jacobian map} for metric
graphs
$$\Pi^{trop}: \mathrm{Out}(F_n)\backslash X_n \to \mathrm{GL}(n, \Z)\backslash  \mathrm{GL}(n, \R)/\mathrm{O}(n).$$
This map can be lifted to the equivariant period map in Equation (\ref{period-map-1})
\begin{equation}\label{period-map-3}
\Pi=\Pi^{trop}: X_n\to \mathrm{GL}(n, \R)/\mathrm{O}(n).
\end{equation}
(For simplicity, we denote the period map here and below by $\Pi$ instead of $\Pi^{trop}$.)

This map can be explained more directly as follows.
We can fix an isomorphism $H_1(R_n, \Z)\cong \Z^n$, where $R_n$ is the rose with $n$ petals.
Then for each marked metric graph $(\ga, \ell, h)$ in $X_n$, the marking $h$ induces
a canonical isomorphism $h_*: H_1(\ga, \Z)\cong \Z^n$, and hence an isomorphism
$h_*: H_1(\ga, \R)\cong \R^n$ by linear extension. By Proposition \ref{lattice-embed}, we have
$\Omega(J)^*\cong H_1(\ga, \R)$ under the period map in Equation \ref{period-map-2}.
This implies that the positive quadratic form $Q$ on $\Omega(J)^*$ induces a positive
definite {\em matrix} on $\R^n$ with respect to the standard basis
 and gives a point in  $\mathrm{GL}(n, \R)/\mathrm{O}(n)$. 
We also note that by Proposition \ref{tropical-flat}, the tropical structure and the principal
polarization $Q$ on the torus $H_1(\ga, \Z)\backslash \Omega(\ga)^*\cong H_1(\ga, \Z)\backslash H_1(\ga, \R)$ 
are all uniquely determined
by the positive definite quadratic form $Q$. 

Note that $\mathrm{GL}(n, \R)$ is a reductive Lie group, and $\mathrm{GL}(n, \R)/\mathrm{O}(n)$
is a symmetric space of nonpositive curvature. Fix an invariant Riemannian
metric on $\mathrm{GL}(n, \R)/\mathrm{O}(n)$, and let $d_{inv}$ be the induced
Riemannian distance function.

If the map $\Pi$ in Equation \ref{period-map-3} were injective, 
then the pull-back metric $\Pi^*d_{inv}$
would give an $\mathrm{Out}(F_n)$-invariant metric on the outer space $X_n$.
Unfortunately,  it is  not injective due to the infinite kernel of $\mathrm{Out}(F_n)\to \mathrm{GL}(n, \Z)$
as pointed out in \S 2, or  by Corollary \ref{failure} when $n\geq 3$.
On the other hand, $\Pi^*d_{inv}$ is a pseudo-metric, i.e., it satisfies all the conditions
of a metric except for the positivity condition: $\Pi^*d_{inv}(x, y)>0$ when $x\neq y$.

Recall that $X_n$ admits the canonical $\mathrm{Out}(F_n)$-invariant incomplete
simplicial metric $d_0$ (Proposition \ref{simplicial-metric}). 
The {\em simple key idea} is to combine $\Pi^*d_{inv}$ with the metric $d_0$. 
For every two points $x, y\in X_n$,
define a {\em distance} $d_1(x, y)$ between them by
\begin{equation}
d_1(x, y)=d_0(x, y)+d_{inv}(\Pi(x), \Pi(y)). 
\end{equation}

\begin{prop}
The function $d_1: X_n\times X_n\to \R_{\geq 0}$ is an $\mathrm{Out}(F_n)$-invariant
metric on $X_n$.
\end{prop}
\begin{proof}
Since the map $\Pi$ is equivariant with respect to the homomorphism
 $\mathrm{Out}(F_n)\to \mathrm{GL}(n, \Z)$, $d_{inv}$ is invariant under $\mathrm{GL}(n, \Z)$,
and $d_0$ is invariant under $\mathrm{Out}(F_n)$, 
it is clear that $d_1$ is invariant under $\mathrm{Out}(F_n)$.
Since $d_0$ is a metric, it is also clear that $d_1$ is a metric on $X_n$. 
\end{proof}

\begin{prop}\label{metric-1}
The metric space $(X_n, d_1)$ is a complete locally compact metric space.
\end{prop}
\begin{proof}
Fix  a basepoint $y_0\in X_n$.
Since $X_n$ is a locally finite simplicial complex with some simplicial faces missing,
it suffices to prove that for any sequence of marked metric graphs 
$y_j=(\ga_j, \ell_i, h_i)$  in $X_n$, 
if $y_j$ converges to a missing simplicial face of a simplex of $X_n$,
then $d_1(x_0, y_j)\to +\infty$ as $j\to +\infty$. 
We note that when $y_j$ approaches a missing simplicial face, the length
of the shortest loop (or nontrivial cycle) of the metric graph $(\ga_j, \ell_i, h_i)$ goes to
$0$ as $j\to +\infty$. This means that the shortest closed geodesic
of the flat tori  $H_1(\ga_j, \Z)\backslash H_1(\ga_j, \R)\cong H_1(R_n, \Z)\backslash 
H_1(R_n, \R)\cong \Z^n\backslash \R^n$ 
goes to zero, i.e., the minimum  value of the positive definite matrix 
$\Pi(y_j)\in \mathrm{GL}(n, \R)/\mathrm{O}(n)$ on nonzero integral points in $H_1(\ga_j, \Z)\cong \Z^n$
goes to 0. By the Mahler compactness theorem for quadratic forms (or lattices),
this implies that $\Pi(y_j)$ is 
a divergent sequence in the symmetric space $\mathrm{GL}(n, \R)/\mathrm{O}(n)$. 
Since the invariant metric $d_{inv}$ on the symmetric space $\mathrm{GL}(n, \R)/\mathrm{O}(n)$ is complete,
$d(\Pi(y_0), \Pi(y_j))\to +\infty$ as $j\to +\infty$. Since  $d_1(y_0, y_j)\geq d_{inv}(\Pi(y_0), \Pi(y_j))$,
the proposition is proved.
\end{proof}

\begin{rem}
{\em The above result shows an important difference between the tropical period
map $\Pi^{trop}$ and  the  period map
for Riemann surfaces, 
$\Pi: \T_g\to \mathfrak h_g$, where $\T_g$ is the Teichm\"uller space of compact 
surfaces of genus $g$ and $\mathfrak h_g= \mathrm{Sp}(2g, \R)/\mathrm{U}(g)$
is the Siegel upper half space. When we pinch a family of hyperbolic Riemann surfaces
along a separating geodesic, i.e., a geodesic which cuts the surface into two connected
components, 
its period matrix does not go to the infinity of $\mathfrak h_g$.
Instead it converges to a point in $(\Omega_1, \Omega_2)\in
\mathfrak h_{g_1} \times \mathfrak h_{g_2}\subset \mathfrak h_g$,
where the limit hyperbolic surface is $\Sigma_{g_1, 1}\cup \Sigma_{g_2, 1}$, $g_1+g_2=g$, 
and $\Omega_1$ is the period of $\Sigma_{g_1}$, 
and $\Omega_2$ is the period of $\Sigma_{g_2}$. In particular, the period does not detect
the punctures in  $\Sigma_{g_1, 1}$ and $\Sigma_{g_2, 1}$.
This implies that the pulled back metric on $\T_g$ by $\Pi$ from the invariant metric 
of the Hermitian symmetric space $\mathfrak h_g$ is not complete.
This metric on $\T_g$ is usually called the Satake metric, probably due to the fact that
a compactification of $\M_g=\Mod_g\backslash \T_g$ was first constructed
via $\Pi$ from the Satake compactification of the Siegel modular variety
$ \mathrm{Sp}(2g, \Z)\backslash  \mathfrak h_g$. 
We note that metric graphs in the outer space $X_n$ do not contain separating edges,
and the above difficulty is avoided.
}
\end{rem}

Recall that a {\em geodesic segment} in a metric space $(X, d)$
is an isometric embedding $i: (a, b) \to (X, d)$ (sometimes its  image is also called
a geodesic segment), and $(X, d)$ is called a {\em geodesic metric space}
if every pair of points  are connected by a geodesic segment.
For many applications, it is important to obtain geodesic metrics.

For every path connected metric space $(X, d)$, there is a natural length structure
$(X, d_\ell)$ \cite[p. 3, 1.4.(a)]{gro} \cite[\S 2.3]{bur}.
Briefly, for every path $\gamma: [a, b]\to X$, define its length
$$\ell(\gamma)=\sup_{t_0< \cdots< t_N}\{\sum_{i=1}^N d(\gamma(t_{i-1}), \gamma(t_i))\},$$
where the $\sup$ is over all partitions $a=t_0< t_1 < \cdots < t_N=b$ of the interval $[a, b]$,
and for every two points $x, y\in X$,  the length metric $d_\ell(x, y)$ is given by 
$$d_\ell(x, y)=\inf_\gamma  \ell(\gamma),$$
where the $\inf$ is over all paths $\gamma$ connecting $x, y$.

In general, a length metric space is not necessarily a geodesic space. 
Let $d_{1, \ell}$ be the length metric on $X_n$ induced from the metric $d_1$
in Proposition \ref{metric-1}.

\begin{thm}\label{main1}
The length metric $d_{1, \ell}$ on $X_n$ is an $\mathrm{Out}(F_n)$-invariant complete
geodesic metric.
\end{thm}
 \begin{proof}
 Since $d_1$ is invariant under $\mathrm{Out}(F_n)$, it is clear that $d_{1, \ell}$ is also
 $\mathrm{Out}(F_n)$-invariant. 
 Since $X_n$ is a locally finite simplicial complex with some simplicial faces missing,
 and $d_1$ is a complete metric on $X_n$, it is clear that $(X_n, d_1)$ is a complete
 locally compact metric space. It is also clear from the definition of $d_1$ and the local
finiteness of the simplicial complex $X_n$ that for every two points $x, y\in X_n$,
$d_{1, \ell}(x, y)<+\infty$. 
 By definition, $d_{1, \ell}\geq d_1$,  and hence $(X_n, d_{1, \ell})$ is also
 a complete locally  compact length space.
 By \cite[Proposition 2.5.22]{bur}, $(X_n, d_{1, \ell})$ is a geodesic length space.
 \end{proof}

 From the above discussion, it is clear that we can also construct other $\mathrm{Out}(F_n)$-invariant
 complete geodesic metrics on $X_n$.
 For example, we can define a metric $d_2$ on $X_n$ by
 $$d_2(x, y)=\sqrt{d_0(x,y)^2+d_{inv}(\Pi(x), \Pi(y))^2},$$
 and obtain a corresponding {\em complete geodesic metric} $d_{2,\ell}$.
 Similarly, we can define
 $$d_\infty(x, y)=\max\{d_0(x, y), d_{inv}(\Pi(x), \Pi(y))\},$$
 and obtain a {\em complete geodesic metric} $d_{\infty, \ell}$.

  As mentioned in the introduction, we would like to construct some complete metrics
  on $X_n$ which resemble Riemannian metrics, for example, its restriction to
  each open simplex of $X_n$ is a Riemannian metric.
   It is not clear whether the above geodesic  metrics $d_{1, \ell}, d_{2, \ell},
  d_{\infty, \ell}$ on $X_n$ enjoy this property.
  
  In the next two sections, we construct two such metrics, which might offer better analogies 
  with the Riemannian metrics on symmetric spaces and Teichm\"uller spaces.

\section{A complete pesudo-Riemannian metric on the outer space via the tropical
Jacobian map}

In the previous section, we use the period map of metric graphs in Equation \ref{period-map-3}
to pull back the Riemannian {\em distance} $d_{inv}$ on the symmetric space $\mathrm{GL}(n, \R)/\mathrm{O}(n)$
in order to define geodesic metrics on $X_n$.
In this section, we want to construct geodesic metrics on $X_n$ which are Riemannian
metrics on each simplex of $X_n$.

Recall that $X=\mathrm{GL}(n, \R)/\mathrm{O}(n)$ is a Riemannian symmetric space of nonpositive
curvature. Let $ds_{inv}^2$ 
be a Riemannian metric tensor on it which is invariant under $\mathrm{GL}(n, \R)$. 
Equivalently, let $\langle \cdot, \cdot \rangle_{inv}$ be the invariant Riemannian metric on $X$.

When restricted to each open simplex $\Sigma$ of $X_n$, the period map 
$\Pi: \Sigma \to X$ is a smooth map. 
 It is natural to pull back the Riemannian metric $ds_{inv}^2$ by $\Pi$.
But the problem is that $\Pi$ is not an embedding, and not even a local immersion when $n\geq 3$, 
by Corollary \ref{failure}. 
(As seen in the proof of Corollary \ref{failure}, 
it can map subsets of positive dimension of $X_n$ to one point).
Therefore, $\Pi^*(ds_0^2)$ is not positive definite, but only semi-positive definite.

To overcome this difficulty, we note that each open simplex $\Sigma$  of 
$X_n$ has a canonical flat Riemannian metric $\langle \cdot, \cdot\rangle_0$
such that each edge has length $\sqrt{2}$. 
Let $ds_0^2$ be the metric tensor of this flat Riemannian metric.

Define a bilinear form on the tangent bundle of $\Sigma$ by
\begin{equation}\label{bilinear}
\langle u, v\rangle_\Sigma=\langle d\Pi(u), d\Pi(v)\rangle_{inv}+\langle u, v\rangle_0,
\end{equation}
where $u, v$ are tangent vectors of $\Sigma$ at any point.
Equivalently, its associated symmetric metric tensor is
\begin{equation}
ds^2=\Pi^* (ds_{inv}^2)+ ds_0^2.
\end{equation}

\begin{lem}
The bilinear form 
$\langle \cdot, \cdot\rangle_\Sigma$ in Equation \ref{bilinear} defines a Riemannian metric on every open simplex $\Sigma$ of $X_n$. 
\end{lem}
\begin{proof}
Since  $ds_0^2$ is positive definite, and $\Pi^* (ds_{inv}^2)$
is semi-positive definite, their sum $ds^2=\Pi^* (ds_{inv}^2)+ ds_0^2$ is
positive definite.
Since $\Pi$ depends smoothly on the simplicial coordinates of $\Sigma$,
this implies that $ds^2$ defines a smooth Riemannian metric on $\Sigma$.
\end{proof}

Since $X_n$ is a simplicial complex and
is a disjoint union of open simplices $\Sigma$,  the metrics $ds^2$ on the open
simplices $\Sigma$ give a {\em stratified
Riemannian metric} on $X_n$. 
We denote it by $ds^2$ or $\langle \cdot , \cdot\rangle$.

A simplicial complex is a stratified space with a {\em smooth structure} in the sense
of \cite[Chapter 1]{pf}, and we want to show that the stratified
Riemannian metric $ds^2$ gives a smooth Riemannian metric
on $X_n$ in the sense of \cite[\S 2.4]{pf}, or rather more directly it is a piecewise
smooth Riemannian metric on $X_n$.

For each simplex $\Sigma$ of $X_n$, let $\R^i$ be the linear space of the same dimension
that contains $\Sigma$, or  is spanned by $\Sigma$.
Let $\overline{\Sigma}$ be the closure of $\Sigma$ in $X_n$, which is  contained in $\R^i$
and but still not a closed simplex. 
A function $f$ or a tensor $T$ on $\overline{\Sigma}$ is called {\em smooth} if there exists
an open subset $U$ of $\R^i$ which contains $\overline{\Sigma}$ such that $f$ or $T$
can be extended to a smooth function or a tensor on $U$.

\begin{prop}\label{smooth-extension}
For every simplex $\Sigma$ of $X_n$, the Riemannian metric $ds^2$ can be extended
to a smooth Riemannian metric on $\overline{\Sigma}$.
\end{prop}
\begin{proof}
It suffices to prove that for every point $p\in {\Sigma}$, there is a neighborhood $U_p$ of $p$ 
in $\R^i$ such that the period map $\Pi: \Sigma
\to X$ extends to a smooth function on $U_p \cup \Sigma \to X$.
Let $(\ga,  h)$ be a marked graph  whose corresponding  simplex $\Sigma_{(\ga, h)}$
 in $X_n$ is $\Sigma$.
Then the point $p$ corresponds to a marked metric graph $(\ga, \ell_p, h)$.
When $p$ lies on the boundary faces of $\sigma$, 
some edge lengths $\ell(e)$ are equal to $0$, but for every 1-cycle,
its total length is positive. By changing all edge lengths slightly
(allowing them to be positive, zero or negative)
 under the condition that
the sum of all edges is equal to 1, we obtain a neighborhood of $U_p$ in $\R^i$ 
such that for every 1-cycle of the graph $\ga$, its length is positive.
Since the period  $Q$ of a metric graph $(\ga, \ell, h)$  only depends on the lengths of 1-cycles,
$Q$ can be extended to $U_p$ and the extended $Q$ is positive definite. 
This gives the required
extension of the period map $\pi$ and
the metric $\Pi^*(ds_{inv}^2)+ds_0^2$ to $U_p\cup \Sigma$.
\end{proof}

\begin{prop}\label{glue-metric}
For every two simplices
$\Sigma_1, \Sigma_2$ with $\overline{\Sigma_1}\cap \overline{\Sigma_2}\neq \emptyset$,
the extended Riemannian metrics $\langle \cdot,  \cdot\rangle_{\Sigma_1}$ on $\overline{\Sigma_1}$
and $\langle \cdot, \cdot \rangle_{\Sigma_2}$ on $\overline{\Sigma_2}$  agree on the intersection $\overline{\Sigma_1}\cap \overline{\Sigma_2}$.
Therefore,  the stratified metric $ds^2$ on $X_n$
is a piecewise smooth Riemannian metric on $X_n$ as a simplicial complex.
\end{prop}
\begin{proof}
Let $\Sigma_3$ be a common simplicial face of $\overline{\Sigma_1}$,
$\overline{\Sigma_2}$. 
For every point $p\in \Sigma_3$, let $\langle \cdot, \cdot \rangle_p$ be the inner product
associated with the metric tensor $\Pi^*(ds_{inv}^2)+ds_0^2$ on $\Sigma_3$.
Then for every tangent vector $v$ to $\Sigma_3$ at $p$, it is clear from the previous
proposition that
$$\langle v, v \rangle_{\Sigma_1}=\langle v, v \rangle_p, \quad  \langle v, v \rangle_{\Sigma_2}=\langle v, v \rangle_p.$$
Therefore, $\langle \cdot,  \cdot \rangle_{\Sigma_1}$ on $\overline{\Sigma_1}$
and $\langle \cdot, \cdot \rangle_{\Sigma_2}$ agree on $\Sigma_3$. 
\end{proof}

With this piecewise smooth Riemannian metric on the stratified space $X_n$, we will introduce a
length metric on $X_n$, following the general setup of \cite{pf}.
 
Specifically, we consider piecewise smooth curves in $X_n$,
i.e., continuous maps $c: [a, b]\to X_n$
such that there is a partition $a=t_0 < t_1 \cdots < t_n=b$,  where
each piece $c:[t_{i-1}, t_i]$ is a smooth curve in a simplex $\overline{\Sigma_i}$.
Define the length of $c$ by
$$\ell(c)=\sum_{i=1}^n\int_{t_{i-1}}^{t_i} \langle c'(t), c'(t) \rangle_{\Sigma_i} dt .$$
For every pair of points $p, q\in X_n$,
define a length function
$$d_{\ell, R}(p, q)=\inf_c \ell (c),$$
where $c$ ranges over all piecewise smooth curves in $X_n$ connecting $p$ and $q$.
We note that the subscript $R$ stands for a Riemannian metric.

\begin{prop}\label{complete-Riemannian}
The  function
$d_{\ell, R}$ defines a complete, locally compact length metric on $X_n$ which is invariant under $\mathrm{Out}(F_n)$ and restricts to a Riemannian distance on each simplex $\Sigma$ of $X_n$.
\end{prop}
\begin{proof}Since $X_n$ is a locally finite
simplicial complex with a piecewise smooth Riemann metric,
and since every two points $p, q$ of $X_n$ are connected by  a piecewise smooth curve,
it is clear that $d_{\ell, R}(p, q)<+\infty$, and if $p\neq q$, then $d_{\ell, R}(p, q)>0$,
and that $d_{\ell, R}$ defines the original topology of $X_n$.

Since the Riemannian metric $ds^2$ is invariant under $\mathrm{Out}(F_n)$, the length
function $d_{\ell, R}$ is also invariant under $\mathrm{Out}(F_n)$.
We need to show that it is complete. This follows from Lemma \ref{complete} below. 
\end{proof}

\begin{lem}\label{complete}
Let $c: [0, 1)\to X_n$ be a continuous path such that when $t\to 1$
the point 
$c(t)$ approaches a missing simplicial face of $X_n$, i.e., 
the total length of some cycle (or loop) of the marked metric graph represented
by $c(t)$ goes to 0. Then the image $\Pi(c(t))$ goes to the infinity of the symmetric
space $X=\mathrm{GL}(n, \R)/\mathrm{O}(n)$ when $t\to 1$.
In particular, the length of the curve $\Pi(c(t))$, and hence of the curve $c(t)$,
 $t\in [0, 1)$, is equal to infinity. 
\end{lem}
 \begin{proof}
When $c(t)$ goes to a missing simplicial face of $X_n$ as $t\to 1$, the period matrix
$Q(t)$ of the metric graph corresponding to $c(t)$ converges to the subspace of
semipositive definite singular symmetric matrices.  
This implies that as $t\to 1$, 
$\Pi(c(t))$ goes to the infinity of the symmetric space $X=\mathrm{GL}(n, \R)/\mathrm{O}(n)$.
Since the symmetric space $\mathrm{GL}(n, \R)/\mathrm{O}(n)$ is a complete Riemannian manifold,
the length of the curve $\Pi(c(t))$, $t\in [0, 1)$, is equal to infinity.

When we compute the length of $c$ using $\Pi^* ds_{inv}^2$, it is equal to the length
of $\Pi(c(t))$. Since the length element of $ds^2$ is greater than or equal to $\Pi^*ds_{inv}^2$,
the length of the curve $c(t)$, $t\in [0, 1)$, is also equal to infinity.
\end{proof}

\begin{thm}\label{main2}
The outer space $X_n$ with the length metric $d_{\ell, R}$
of the piecewise Riemannian metric $ds^2$
is a complete geodesic space, and every two points are connected by
a geodesic which is a piecewise smooth curve in $X_n$. 
\end{thm}
\begin{proof}
Since $X_n$ is locally compact and $d_{\ell, R}$ is a length distance, it follows
from \cite[Proposition 2.5.22]{bur} \cite[Theorem 2.4.15]{pf} that $(X_n, d_{\ell, R})$
is a geodesic space, i.e., every two points are connected by a geodesic segment. 
Since $(X_n, d_{\ell, R})$ is the length structure induced from a piecewise smooth
Riemann metric of a locally finite simplicial complex, it follows from the usual argument
in Riemannian geometry (see \cite{lee} for example) that every geodesic segment  in $X_n$
is a piecewise smooth curve.
\end{proof}

\begin{rem}
{\em Once we have endowed $X_n$ with an invariant piecewise Riemannian metric, 
it is  natural to study some geometric invariants. For example, we can study the growth
of balls in $X_n$, and the spectral theory of $X_n$ and $\mathrm{Out}(F_n)\backslash X_n$ 
with respect to suitable defined Laplacian operators.  
As discussed in \cite{pf}, there is little known for analysis on 
general stratified spaces. 
}
\end{rem}

\section{Another complete pseudo-Riemannian metric on the outer space via the lengths of pinching loops}

In this section, we construct another Riemannian metric 
on the outer space $X_n$ whose asymptotic behaviors
are explicitly given. This is motivated by the McMullen metric for Teichm\"uller space in \cite{mc}. 

Fix a sufficiently small $\varepsilon>0$, for example, $\varepsilon < \frac{1}{6n}$.
Define a modified edge length of 1-cycles $\sigma$ of metric graphs $(\ga, \ell)$:
$\ell_\varepsilon(\sigma)$ is a positive smooth function of $\ell(\sigma)$ satisfying the
conditions
\begin{equation}
\ell_\varepsilon(\sigma)=
\begin{cases} \varepsilon, \text{\ if\ } \ell(\sigma)\geq 2 \varepsilon, \\
\ell(\sigma), \text{\ if\ } \ell(\sigma)\leq  \varepsilon.
\end{cases}
\end{equation}

For each simplex $\Sigma$ in $X_n$, define a Riemannian metric by
\begin{equation}
ds^2_\varepsilon= ds_0^2+\sum_{\sigma\text{\ with\ } \ell(\sigma)< \varepsilon} (\frac{ d \ell_\varepsilon(\sigma)}{\ell_\varepsilon(\sigma)})^2,
\end{equation}
where the sum is over all cycles whose lengths are less than $\varepsilon$.

It is clear that for each simplex $\Sigma$, $ds^2_\varepsilon$ defines a Riemannian metric.
This defines a stratified Riemannian metric on $X_n$. 
We note that a sequence of marked metric graphs converges to some missing
simplicial complexes of $X_n$ if and only if there are cycles (or loops) $\sigma$
whose lengths $\ell(\sigma)$ converge to 0. Such loops are called {\em pinching loops}
of metric graphs
and they correspond to pinching geodesics on hyperbolic surfaces.

By the proof of Proposition \ref{smooth-extension},
the length function $\ell(\sigma)$ and hence the modified length function
$\ell_\varepsilon(\sigma)$ is a smooth function on the closure $\overline{\Sigma}$
in $X_n$, and the Riemannian metric $ds^2_\varepsilon$
extends to a smooth Riemannian metric on $\overline{\Sigma}$. 
Then by the proof of Proposition \ref{glue-metric},
we can prove the following proposition.

\begin{prop}
The Riemannian metric $ds^2_\varepsilon$ on the stratified space
 $X_n$ defines an $\mathrm{Out}(F_n)$-invariant piecewise smooth
Riemannian metric on $X_n$. 
\end{prop}

By arguments similar to those of the proofs of Proposition
\ref{complete-Riemannian} and Theorem \ref{main2}, we can prove the following.

\begin{thm}\label{main4}
The piecewise smooth Riemannian metric $ds^2_\varepsilon$
on the outer space $X_n$ defines a complete geodesic metric $d_{\ell, R, \varepsilon}$
on $X_n$ which is invariant under $\mathrm{Out}(F_n)$
such that every two points are connected by a piecewise smooth geodesic.
\end{thm}

\section{Finite Riemannian volume of the quotient $\mathrm{Out}(F_n)\backslash X_n$}

For a Riemannian manifold, a natural invariant is its volume.
As mentioned before, with respect to many natural metrics on $\T_{g,n}$,
the quotient $\Mod_{g, n}\backslash \T_{g,n}$ has finite volume.
Note that  the locally symmetric space
$\mathrm{GL}(n, \Z)\backslash \mathrm{GL}(n, \R)/\mathrm{O}(n)$ has infinite volume.
Since elements of $\mathrm{GL}(n, \Z)$ have determinant equal to $\pm 1$, 
$\mathrm{GL}(n, \Z)$ also
acts on $\mathrm{SL}(n, \R)/\mathrm{SO}(n,\Z)$, which is a symmetric space of 
noncompact type,
and the quotient space $\mathrm{GL}(n, \Z)\backslash \mathrm{SL}(n, \R)/\mathrm{SO}(n,\Z)$
 has finite volume.
Or equivalently,  we note that $\mathrm{GL}(n, \Z)$ contains $\mathrm{SL}(n, \Z)$ as a subgroup of finite index,
and the quotient $\mathrm{SL}(n, \Z)\backslash \mathrm{SL}(n, \R)/\mathrm{SO}(n,\Z)$ has finite volume.

If we pursue the analogy between the three pairs of transformation groups
 $(\mathrm{Out}(F_n), X_n),$ $(\T_g, \Mod_g)$,
$(\mathrm{SL}(n, \Z), \mathrm{SL}(n, \R)/\mathrm{SO}(n))$, 
then we expect that $\mathrm{Out}(F_n)$ is a lattice
and the following conjecture seems natural.

\begin{con}
The volume of $\mathrm{Out}(F_n)\backslash X_n$ with respect to either of 
the Riemannian metrics $ds^2$, $ds^2_{\varepsilon}$ defined in the previous section
is finite.
\end{con} 

By definition, the volume of a simplicial complex with respect to a piecewise Riemannian metric
is the sum of volumes of the top dimensional simplices.
By Proposition \ref{finite-cell}, there are only finitely $\mathrm{Out}(F_n)$-orbits of simplices.
Therefore, it suffices to decide whether for each top dimensional simplex $\Sigma$
of $X_n$, the volume of $\Sigma$ with respect to $ds^2$ or $ds^2_\varepsilon$
is finite.

\begin{prop}\label{finite-vol-1}
When $n=2$, the volume of every 2-simplex $\Sigma$  in $X_n$ with respect to the Riemannian
metric $ds^2$ is finite, and hence the volume of $\mathrm{Out}(F_n)\backslash X_n$ 
with respect to the Riemannian metric $ds^2$  is finite. 
\end{prop}
\begin{proof}
To prove this proposition, we need to compute $\Pi^*(ds^2_{inv})$.
Let $(\ga, h)$ be a marked graph corresponding a 2-dimensional simplex $\Sigma$ of
$X_2$. Then $\ga$ has 3 edges $e_1, e_2, e_3$ and 2 vertices, 
with each edge connecting these two vertices. Denote the edge lengths by
$a, b, c$. Then they are characterized by the following conditions:
(1) $a, b, c\geq 0$, (2) $a+b+c=1$, (3) $a+b, a+c, b+c>0$.
In particular, the vanishing of one edge length is allowed, but not two edges simultaneously.
Therefore, $\Sigma$ has 3 missing simplicial faces of dimension 0, i.e., missing vertices,
with each corresponding to the vanishing of two edge lengths.
It suffices to show that the area of the corner near each missing vertex is zero. 
Assume that $a, b=0$ at this corner.
Let  $\varepsilon$ be a small positive number
and define a corner $\Omega_\varepsilon$ by $a+b\leq \varepsilon$.
The period matrix of the marked metric graph is
$$Q=\begin{pmatrix} a+b & b \\ b & 1-a\end{pmatrix},$$
and its inverse is
$$Q^{-1}=\frac{1}{ab+(a+b)(1-a-b)} \begin{pmatrix} 1-a & -b \\ -b & a+b\end{pmatrix}.$$
By \cite[\S 4.1.2]{ter}, the invariant Riemannian metric of the symmetric
space $X=\mathrm{GL}(n,\R)/\mathrm{O}(n)$ is given by the metric tensor
$$ds^2_{inv}=\text{Tr} ( (Y^{-1} d Y)^2), $$
where $Y\in X$ is a symmetric matrix in $X$.
Therefore, the induced  Riemannian metric on $\Sigma$
is $ds^2=ds^2_0+\text{Tr} ( (Q^{-1} d Q)^2)$.
We estimate the growth of $ds^2$ when $a, b\to 0$.
The term $ds_0$ is bounded and we ignore it.
Since $\frac{1}{ab+(a+b)(1-a-b)}$ is comparable with $\frac{1}{a+b}$,
and $$d Q=\begin{pmatrix} ad +db & db \\a b & - da\end{pmatrix},$$
by a direct computation, 
 the area form of the Riemannian metric $ds^2$
is bounded from above by a multiple of
$\frac{1}{(a+b)^{3/2}} da\wedge d b$.
By dividing the  region $\Omega_\varepsilon$ into two subregions according to 
$(I): a\leq b$, $(II): b\leq a$, we can show easily that 
the area of $\Omega_\varepsilon$ with respect to $ds^2$ is finite.  
\end{proof}

\begin{rem}
{\em 
In the above estimate, the assumption that $n=2$ was used for the explicit computation.
 It is conceivable that a more detailed computation will allow one to prove the general case. 
}
\end{rem}

Using the explicit form of the metric $ds^2_\varepsilon$, we can also show that
in the corner $\Omega_\varepsilon$ of a simplex $\Sigma$ as above,
$ds^2_\varepsilon$ is comparable with 
$\frac{(d a+ d b)^2}{(a+b)^2}+ da^2+ db ^2$, and that the area form is bounded
by  a multiple of $\frac{1}{a+b} ad \wedge db$.
Therefore, we can prove that the area of the corner is finite
with respect to the Riemannian metric $ds^2_\varepsilon$
and prove the following result.

\begin{prop}\label{finite-vol-2}
When $n=2$, the volume of every 2-simplex in $X_n$ with respect to the Riemannian
metric $ds^2_\varepsilon$ is finite, and hence the volume of $\mathrm{Out}(F_n)\backslash X_n$
with respect to the Riemannian metric $ds^2_\varepsilon$ is finite. 
\end{prop}

\begin{rem}
{\em When $n=2$, there is a natural identification of $X_2$ with the Poincare disk \cite{vog}.
It is probably true that the metric $ds^2$ is quasi-isometric to the Poincare metric.
}
\end{rem}


In this paper, we have introduced several
 geodesic metrics $d_{1, \ell},$ $ d_{2, \ell},$ $ d_{\infty, \ell},$ $
d_{\ell, R},$ $ d_{\ell, R, \varepsilon}$ on $X_n$. We note that
$d_{1, \ell},$ $ d_{2, \ell},$ $ d_{\infty, \ell}$ are clearly quasi-isometric to each other. 
The following problem seems to be natural and tempting.

\begin{prob}
Decide whether all the above $\mathrm{Out}(F_n)$-invariant, complete geodesic metrics 
$d_{1, \ell}$,
$d_{\ell, R}, d_{\ell, R, \varepsilon}$ on $X_n$ are quasi-isometric. 
\end{prob}


\end{document}